%% file: main.tex
\documentclass[a4paper,12pt,hidelinks]{article}
\usepackage[utf8]{inputenc}
\usepackage[margin=2.5cm]{geometry}
\usepackage{amsmath}
\usepackage{amssymb}
\usepackage{amsthm}
\usepackage{tikz-cd}
\usepackage{mathtools}
\usepackage{comment}
\usepackage{faktor}
\usepackage{enumitem}
\usepackage{thmtools}
\usepackage{thm-restate}
\usepackage{subfig}
\usepackage{float}
\usepackage{setspace}
\usepackage{caption}

\usepackage{hyperref}
\usepackage[nameinlink]{cleveref}

\numberwithin{equation}{section}
\theoremstyle{plain}
  \newtheorem{theorem}[equation]{Theorem}
  \newtheorem{proposition}[equation]{Proposition}
  \newtheorem{corollary}[equation]{Corollary}
  \newtheorem{lemma}[equation]{Lemma}
  
  \newtheorem{conjecture}[equation]{Conjecture}

  \theoremstyle{definition}
  \newtheorem{definition}[equation]{Definition}
  
  \newtheorem{notation}[equation]{Notation}

  \newtheorem{example}[equation]{Example}

\newcommand{\ZZ}{\mathbb{Z}}

\newcommand{\NN}{\mathbb{N}}

\newcommand{\kk}{\Bbbk}

\newcommand{\g}{\mathfrak{g}}

\DeclareMathOperator{\ad}{ad}

\DeclareMathOperator{\Vir}{Vir}
\DeclareMathOperator{\PsiS}{\Tilde{S}}

\newcommand{\mc}{\mathcal}
\newcommand{\mf}{\mathfrak}
\newcommand{\mb}{\mathbf}

\DeclareMathOperator{\GKdim}{GKdim}

\DeclareMathOperator{\gr}{gr}
\DeclareMathOperator{\Der}{Der}

\DeclareMathOperator{\Sa}{S}
\DeclareMathOperator{\Ua}{U}
\DeclareMathOperator{\ord}{ord}
\newcommand{\ideal}{\unlhd}
\newcommand{\UW}{\Ua(\W{-1})}

\newcommand{\W}[1]{W_{\geq #1}}

\newcommand{\del}{\partial}

\newcommand{\oo}[3]{\Omega^{(#1)}_{#2, #3}}
\DeclareMathOperator{\sym}{sym}

\DeclareMathOperator{\proj}{proj}

\makeatletter
\newcommand{\subjclass}[2][1991]{%
  \let\@oldtitle\@title%
  \gdef\@title{\@oldtitle\footnotetext{#1 \emph{Mathematics Subject Classification.} #2}}%
}
\newcommand{\keywords}[1]{%
  \let\@@oldtitle\@title%
  \gdef\@title{\@@oldtitle\footnotetext{\emph{Key words and phrases.} #1.}}%
}
\makeatother

\title{The Kernel and Image of Orbit Homomorphisms for the Witt Algebra}
\author{Tuan Anh Pham and James Timmins}

\date{}

\begin{document}
\maketitle

\begin{abstract}
    \noindent
    The Witt algebra $\W{-1}$ is the Lie algebra of algebraic vector fields on a line. We investigate the two-sided ideal structure of its universal enveloping algebra, by studying the orbit homomorphisms $\Psi_n: \Ua(\W{-1}) \rightarrow T_n$, an infinite family of homomorphisms to noncommutative Noetherian algebras. The orbit homomorphisms lift primitive ideals from solvable Lie algebras to $\Ua(\W{-1})$, thereby playing a central role in the orbit method for the Witt algebra.

    We prove that the kernel of any orbit homomorphism is generated by an infinite set of differentiators as a one-sided ideal, whilst being generated by any single element of this set as a two-sided ideal. One consequence is an explicit description of primitive and semi-primitive ideals of $\Ua(\W{-1})$ corresponding to one-point local functions. We also prove that the image $B_n$ of the $n$th orbit homomorphism is both non-Noetherian and birational to the Noetherian algebra $T_n$. On the other hand, the degree zero subring of $B_n$ is left and right Noetherian, and we conjecture that the same holds for $\Ua(\W{-1})$. 
\end{abstract}

\input{introduction.tex}
\input{preliminaries.tex}

\input{image.tex}
\input{kernel.tex}

\bibliographystyle{amsalpha}
\bibliography{bibliography.bib}
\vspace{20pt}
{\small \noindent School of Mathematics, The University of Edinburgh, Edinburgh EH9 3FD, United Kingdom\\
\textit{Email address}: \texttt{tuan.pham@ed.ac.uk}\\ \textit{Email address}: \texttt{james.timmins@ed.ac.uk}
}

\end{document}

%% file: introduction.tex
\section{Introduction}
    The (one-sided) \emph{Witt algebra} $\W{-1}$ is the Lie algebra of algebraic 
    vector fields on a line. Let $\kk$ be a field of characteristic zero. The $\W{-1}=\kk[t]\del$ is the Lie algebra of derivations on $\kk[t]$, where $\del = \frac{d}{dt}$. The Lie bracket is given by 
\[[f \del, g \del] = (fg'-f'g) \del.\]

    The one-sided Witt algebra is a simple subalgebra of the Witt algebra $W$, the Lie algebra of derivations on $\kk[t,t^{-1}])$, whose unique (one-dimensional) central extension is the famous Virasoro algebra $\Vir$. The study of representations of the Virasoro algebra is a fundamental topic of mathematical physics. Since representations of any Lie algebra $\g$ are precisely modules for the universal enveloping algebra $\Ua(\g)$, the algebra $\UW$ therefore plays an integral role in representation theory and physics. However, many ring-theoretic properties of $\UW$, such as the structure of its ideal lattice, remain mysterious. 

    In this article we study the enveloping algebra of the Witt algebra via an infinite family of algebra homomorphisms to noncommutative Noetherian algebras, named orbit homomorphisms. The orbit homomorphisms generalize the inclusion of $W_{\geq -1}$ into the first Weyl algebra $A_1 = \kk[t,\del]$.
    
    To describe the higher orbit homomorphisms, notice that the one-sided Witt algebra has a basis $\{e_i \mid i \geq -1\}$ where $e_i = t^{i+1}\del$. The Lie subalgebras $\W{n} = \kk\{e_i | i \geq n\}$ have finite dimensional subquotients $\g_n = \W{0}/\W{n}$. The Lie algebra $\g_n$ has a basis $\{v_{i} \mid 0 \leq i \leq n-1\}$ where $v_i = e_i + \W{n}$, and is solvable.
    
    The $n$th \emph{orbit homomorphism} for $\W{-1}$ is defined by \cite[Theorem 5.7]{pham2025orbit} as 
    \begin{align*}
    \Psi_n: \UW &\rightarrow A_1 \otimes_{\kk} \Ua(\mf{g}_n),\\ 
    f(t)\del &\mapsto f(t)\del + \sum_{i =0}^{n} \frac{f^{(i+1)}(t)}{(i+1)!} v_i,
\end{align*}
    where $f^{(i)}$ denotes the $i$th derivative of $f$. We call $\Psi_n$ an \emph{orbit homomorphism} due to its role in generalising the orbit method from $\mf{g}_n$ to $\W{-1}$ in \cite{pham2025orbit}, described later in the Introduction. From now on, we denote $T_n = A_1 \otimes_{\kk} \Ua(\mf{g}_n)$ and $B_n = \Psi_n(\UW)$.\vspace{-5pt}\\

    In this article we completely determine the kernels of the orbit homomorphisms, and describe properties of their images. This generalizes work of Conley and Martin, who studied $\Psi_0$ and $\Psi_1$ in \cite{conley2007}.
    
    Homomorphisms from $\UW$ have been studied before: in \cite{SierraWalton}, Sierra and Walton used algebro-geometric tools to construct a surjective homomorphism from $\UW$ to a non-Noetherian, birationally commutative algebra, which immediately implies that $\UW$ is not Noetherian. This result was a significant step towards understanding a fundamental connection between Lie theory and ring theory. Whilst the famous PBW theorem shows that any finite-dimensional Lie algebra has a Noetherian universal enveloping algebra, the converse statement is a longstanding conjecture (its first known appearance in the literature is \cite[page 216]{amayo1974infinite}). The result of \cite{SierraWalton} is a crucial step towards proving the conjecture; among its consequences are that all simple infinite-dimensional $\mathbb{Z}^n$-graded Lie algebras have a non-Noetherian enveloping algebra, see \cite{andruskiewitsch2025noetherian}.

    In essence, the homomorphism constructed in \cite{SierraWalton} is very close to the first orbit homomorphism $\Psi_1$, and in \cite{sierra2016maps}, the same authors study $\Psi_1$ systematically. There, Sierra and Walton showed that the subalgebra $\Psi_1 (\Ua(\W{1}))$ of $B_1$ is non-Noetherian (but satisfies the ascending chain condition on two-sided ideals) and effectively showed that $\ker \Psi_1$ is infinitely generated as a left or right ideal. In a similar vein, Conley and Martin showed in \cite{conley2007} that $\ker \Psi_0$ and $\ker \Psi_1$ are principal (two-sided) ideals of $\Ua(\W{-1})$, and calculated their generators. On the other hand, $\Psi_0(\Ua(\W{-1})) = \kk \oplus A_1 \del$ is a left and right Noetherian ring.
    
    These results motivate our study of the images and kernels of the higher orbit homomorphisms. Our ultimate aim is to shed more light on the ring-theoretic properties of $\UW$, many of which remain to be discovered -- for example, it is unknown if ascending chains of two-sided ideals must terminate, see \Cref{UW ACC conjecture}.\vspace{-5pt}\\
    
    Our main result on the image of the $n$th orbit homomorphism is that although it is not Noetherian, it is ``birationally Noetherian''.

\begin{restatable}{theorem}{largeimage} \label{Main result large image}
Let $n \geq 2$. The image $B_n = \Psi_n(\UW)$ contains a principal ideal of $T_n$, namely $T_n v_{n-1}^2$. The ideal $T_n v_{n-1}^2$ is a principal two-sided ideal of $B_n$ but is not finitely-generated as a left or right ideal.
\end{restatable}

From \Cref{Main result large image} we deduce that $B_n$ is not Noetherian, but is birational to (has the same ring of quotients as) the Noetherian ring $T_n$.

\begin{restatable}{corollary}{birationalimage} \label{Main result birational image}
    Let $n \geq 1$. The image of $\Psi_n: \UW \rightarrow T_n$ is a non-Noetherian ring of GK-dimension $n+2$ which is birational to $T_n$.
\end{restatable}

\Cref{Main result large image} also plays a crucial role in a surprising property of the homomorphic images $B_n$; although $B_n$ is a non-Noetherian ring, the subring of degree zero elements (see \Cref{T grading def}) is Noetherian.

\begin{restatable}{theorem}{degreezeroNoetherian} \label{Main result degree zero Noetherian}
    Let $n \geq 0$. The ring $(B_n)_0$ is a left and right Noetherian domain.
\end{restatable}

We further conjecture that the same holds for $\UW$, see \Cref{Witt degree zero conjecture}.

    We prove the above results in \Cref{sec:image}, with the proofs of \Cref{Main result large image} and \Cref{Main result birational image} appearing in subsections \ref{The step elements} to \ref{Birationality}, whilst sub\cref{The subring Bn0 is Noetherian} is devoted to the proof of \Cref{Main result degree zero Noetherian}.\vspace{-5pt}\\

    In \cite{iyudu2020enveloping} it was proved that $\UW$ has just-infinite growth. Since the images $B_n$ are of unbounded GK-dimension by \Cref{Main result birational image}, it follows that the intersection of the kernels of all orbit homomorphisms is zero. For this reason, certain properties of ``large enough'' $B_n$ should be reflected in $\Ua(\W{-1})$. To be more precise, a (non-commutative polynomial) equation in $\UW$ holds if and only if the corresponding equation holds in $B_n$ for all large enough $n$.
    
    The second main result of this article gives explicit generators for the kernels of orbit homomorphisms, via the \emph{differentiators} defined in \cite{liu2015class} and \cite{billig2016classification}. For $m \geq 0, s \geq -1$ and $k \geq m-1$, the differentiator $\oo{m}{k}{s} \in \UW$ is a quadratic element with the closed form,
    \[ \oo{m}{k}{s} = \sum_{i=0}^{m} (-1)^i \binom{m}{i} e_{k-i} e_{s+i}.\]

    The differentiators have been studied before: \cite{billig2016classification} proves a formula relating products of differentiators, and uses it to show that every finite length cuspidal module for the Witt algebra (those with $e_0$-weight spaces of bounded dimension) is annihilated by some differentiator.
    
    We prove that the kernel of any orbit homomorphism is generated by a single differentiator.
    
    \begin{restatable}{theorem}{kernel} \label{Main result kernel}
    Let $n \geq 1$. The kernel of $\Psi_n$ is the principal two-sided ideal
    \[ \ker \Psi_n = (\oo{2n+2}{2n+1}{-1}).\]
    However, $\ker \Psi_n$ is not finitely-generated as either a left or right ideal of $\Ua(\W{-1})$.
    \end{restatable}
    
    Further, we classify the two-sided ideals generated by a single differentiator: they are either zero or the kernel of an orbit homomorphism. We also compute explicit one-sided generating sets, which have two remarkable properties: they consist entirely of differentiators, and any non-zero differentiator in the generating set generates the kernel as a two-sided ideal.

    \begin{restatable}{corollary}{generickernel} \label{Main result generic kernel} 
    Let $n \geq 1$ and $m \in \{2n+1, 2n+2\}$. Let $k_0 \geq m-1$ and $s_0 \geq -1$ be integers such that $\oo{m}{k_0}{s_0} \neq 0$. Then
    \[ (\oo{m}{k_0}{s_0}) = \ker \Psi_n = \Ua(\W{-1}) \{\oo{m}{k}{s} \mid k \geq m-1, s \geq -1\}.\]
    \end{restatable}

    By the result of \cite{billig2016classification} mentioned above, every finite length cuspidal module for the Witt algebra is annihilated by a differentiator. Combining this with \Cref{Main result generic kernel} shows that every finite length cuspidal module is a $B_n$-module for large enough $n$, see \Cref{cuspidal module cor}.
    
    \Cref{Main result kernel} and \Cref{Main result generic kernel} are proved in \Cref{sec:kernel}. They give an infinite class of examples of principal two-sided ideals that are not finitely-generated left or right ideals of $\Ua(\W{-1})$. These previously unknown examples are evidence for the hypothesis that two-sided ideals of $\Ua(\W{-1})$ are usually ``big''. Consequently, our results may be regarded as evidence for a conjecture on the fundamental ideal structure of the ring $\Ua(\W{-1})$.

    \begin{conjecture}[{\cite[Conjecture 1.3]{petukhov2020ideals}, \cite[Question 0.11]{sierra2016maps}}] \label{UW ACC conjecture}
        Two-sided ideals of $\Ua(\W{-1})$ satisfy the ascending chain condition. That is, all two-sided ideals of $\Ua(\W{-1})$ are finitely-generated.
    \end{conjecture}

    Another consequence of \Cref{Main result kernel} is that $B_n$ can be realised as successive extensions of $B_1$ by the ideals $T_2v_1^2, \dots, T_{n}v_{n-1}^2$, see \Cref{B_n kernel thm}.\vspace{-5pt}\\
    
    As mentioned, the cases $n=0$ and $n=1$ of \Cref{Main result kernel} are due to \cite{conley2007}, where they were used to obtain generators for the annihilators of tensor density modules. Similarly, the higher orbit homomorphisms play a fundamental role in the representation theory of the Witt algebra, which we now describe.

    The orbit homomorphisms $\Psi_n$ are defined by the first author in \cite[Theorem 5.7]{pham2025orbit}, during the construction of a very large class of primitive ideals (annihilators of irreducible representations) of $\UW$, known as the \emph{orbit method} for the Witt algebra. Specifically, for every element of the dual vector space,
    \[ \chi \in \W{-1}^* = \mathrm{Hom}_\kk(\W{-1}, \kk),\]
    a primitive ideal $Q_\chi$ of $\UW$ is constructed. 
    
    In fact, the orbit method for the Witt algebra associates a primitive ideal of $\UW$ to every \emph{Poisson primitive} ideal of $\gr \UW$, which are classified in \cite{petukhov2023poisson}. We refer the reader to \cite{pham2025orbit} for full details. It is currently unknown if this construction exhausts all primitive ideals of $\UW$.
    
    The primitive ideal $Q_\chi$ corresponding to $\chi \in \W{-1}^*$ is non-zero if and only if $\chi$ is a \emph{local function}, defined in \cite{petukhov2023poisson}. Local functions are spanned by the \emph{one-point local functions}, which are of the form 
    \[ \chi = \chi_{x; \alpha_0, ...,\alpha_n}: \W{-1} \mapsto \kk, \quad f \partial \mapsto \alpha_0 f(x) + \alpha_1 f'(x) + ... +\alpha_n f^{(n)}(x),\]
    where $n \geq 0$ is called the order of $\chi$, and $x, \alpha_0, \alpha_1, ...,\alpha_n \in \kk$ with $\alpha_n \neq 0$.
    
    When $\chi$ is a one-point local function of order $n$, $\chi$ descends to $\bar{\chi} \in \g_n^*$. The remarkable orbit method for finite-dimensional solvable Lie algebras due to Dixmier, Conze, Duflo and Rentschler \cite[Theorem 6.5.12]{dixmier1996enveloping}, states that $\bar{\chi}$ corresponds to a primitive ideal $Q_{\bar{\chi}}$ of $\Ua(\g_n)$; moreover, the fibres of the correspondence are precisely coadjoint orbits of $\g_n^*$.\vspace{-5pt}\\
    
    The fundamental lifting result \cite[Theorem 1.4]{pham2025orbit} is that 
    \[ Q_{\chi} = \Psi_n^{-1}(A_1\otimes_\kk Q_{\bar{\chi}}).\]
    That is, the primitive ideals of $\UW$ corresponding to one-point local functions of order $n$ are exactly those lifted from primitive ideals of $\Ua(\g_n)$ via the $n$th orbit homomorphism $\Psi_n$. 
    
    As a consequence, the orbit homomorphisms are ``universal homomorphisms'' for one-point local functions, meaning that (see \cite[Corollary 6.11]{pham2025orbit}),
\begin{align*}
    \ker \Psi_{n} &= \bigcap \{Q_\chi \mid \chi \text{ a one-point local function of order} \leq n  \} \\ 
    &=\bigcap \{Q_\chi \mid \chi \text{ a one-point local function of order\;} n \}.
\end{align*}
    
    Since \Cref{Main result kernel} computes kernels of orbit homomorphisms, we obtain explicit descriptions of certain primitive and semi-primitive ideals of $\UW$, see sub\cref{Primitive ideals}.

\begin{restatable}{corollary}{evenoddannihilators} \label{Main result even odd annihilators}
    Let $\chi$ be a one-point local function on $\W{-1}$ of order $n \geq 1$.
    
    If $n$ is even, then
    \[ Q_\chi = \ker \Psi_{n} = (\oo{2n+2}{2n+1}{-1}). \]
    Thus $(\oo{2n+2}{2n+1}{-1})$ is a primitive ideal of $\Ua(\W{-1})$.
    
    If $n$ is odd, then
    \[ Q_\chi \supseteq \ker \Psi_{n} = (\oo{2n+2}{2n+1}{-1}) = \bigcap \{Q_\eta \mid \eta \text{ a one-point local function of order} \leq n  \}. \]
    Thus $(\oo{2n+2}{2n+1}{-1})$ is a semi-primitive ideal of $\Ua(\W{-1})$.
\end{restatable}

    Primitive ideals corresponding to arbitrary local functions $\chi$ are also studied in \cite{pham2025orbit}, by taking tensor products of the orbit homomorphisms. Understanding the kernel and image of such maps would give similar results to \Cref{Main result even odd annihilators} for arbitrary local functions. This is a topic of ongoing research.

    We believe that all of the main results in this article hold for the Witt algebra $W$ -- the orbit homomorphisms $\Psi_n$ are defined in this larger context. However, the homomorphism $\Psi_\infty$ given in \Cref{Psi formula}, which is a convenient technical tool throughout our arguments, is not well-defined on $\Ua(W)$. This prevents our methods from directly applying to $W$.\vspace{-5pt}\\

The organisation of the paper is as follows. In \Cref{Preliminaries}, we define filtrations and gradings on $\Ua(\W{-1})$ and related rings, and recall results regarding $\Psi_n$ and its associated graded map from \cite{pham2025orbit}. We also define the differentiators $\oo{m}{k}{s}$ and prove some useful commutator relations. In \Cref{sec:image}, we construct an operator $\PsiS$ on the images $B_n$ based on a step algebra element $S$ of $\Ua(\W{-1})$, and show \Cref{Main result large image} and \Cref{Main result birational image}. We then prove \Cref{Main result degree zero Noetherian} in sub\cref{The subring Bn0 is Noetherian}. In \Cref{sec:kernel}, we calculate the kernel of $\Psi_n$ using the method of \cite{conley2025annihilators}. We prove \Cref{Main result kernel} in sub\cref{Proof of kernel thm}, deducing some interesting consequences. We end the paper by proving \Cref{Main result generic kernel} and \Cref{Main result even odd annihilators} in sub\cref{Generation of ker Psi_n,Primitive ideals} respectively.\\

\textit{Acknowledgements.} We thank Susan J. Sierra and Charles H. Conley for interesting discussions and
questions. This research was part of the first author’s PhD research at the University of Edinburgh. The second author was supported by EPSRC grant EP/T018844/1, and the ERC under the EU Horizon 2020 research and innovation programme, under grant agreement 948885.

%% file: preliminaries.tex
\section{Preliminaries} \label{Preliminaries}
We fix a base field $\kk$ of characteristic $0$ throughout the paper. All algebras mentioned will be unital $\kk$-algebras and all vector spaces (in particular, Lie algebras) are assumed to be defined over $\kk$. Unless otherwise specified, tensor products are over $\kk$, so we abbreviate $\otimes_\kk$ to $\otimes$. The symbol $\NN^d$ denotes $d$-tuples of non-negative integers.

\subsection{Lie algebras and homomorphisms}

We first introduce our primary Lie algebra of interest. 

\begin{definition}
    Let $\W{-1} = \Der(\kk[t]) = \kk[t]\del$ be the \emph{one-sided Witt algebra} where $\del = \frac{d}{dt}$. Explicitly, $\W{-1}$ has the Lie bracket
    \begin{align*}
        [f \del, g \del] = (fg'-f'g)\del \quad \text{ for all } f, g \in \kk[t].
    \end{align*} 
\end{definition}

The one-sided Witt algebra $\W{-1}$ is also the Lie algebra of algebraic vector fields on $\kk$.

\begin{definition}
    For all $i \geq -1$, let $e_i = t^{i+1}\del \in \W{-1}$. These form a countable basis for $\W{-1}$, with the Lie bracket
      \[[e_i, e_j] = (j-i) e_{i+j}.\]  
\end{definition}

\begin{definition}
For any $n \geq -1$, define $\W{n} = \kk\{e_i | i \geq n\}$, which is a Lie subalgebra of $\W{-1}$.
\end{definition}

It is straightforward that $\W{n}$ is a Lie ideal of $\W{0}$, and hence we can form a quotient.

\begin{definition}
     For $n \geq 0$, we define the Lie algebra $\g_n = \W{0}/\W{n}$. We define $v_{i} = e_i + \W{n}$, for $0 \leq i \leq n-1$.
\end{definition}

Note that $\W{-1}$ is a simple Lie algebra, however, the subquotients $\g_n$ are finite-dimensional solvable Lie algebras.

\begin{definition}\label{def:ringT}
    Let $A_1$ be the first Weyl algebra, generated by $t, \del$ subject to the relation $\del t - t \del = 1$. We will abuse notation and write $A_1 = \kk [t, \del]$. For all $n \geq 0$, define the rings 
     \begin{align*}
         T_n = \kk[t, \partial] \otimes_{\kk} \Ua(\g_n),
     \end{align*}
     and additionally,
     \begin{align*}
         T_\infty =  \kk[t, \partial] \otimes_{\kk} \Ua(\W{0}).
     \end{align*}
\end{definition}

We can now give the family of homomorphisms central to this article, which were defined by the first author in \cite{pham2025orbit}.
\begin{proposition}[{\cite[Theorem 5.7]{pham2025orbit}}] \label{Psi formula}
    There are Lie algebra homomorphisms
    \[
    \psi_n: \W{-1} \to T_n, \quad f\del \mapsto f\del + \sum_{i=0}^{n-1} \frac{f^{(i+1)}}{(i+1)!} v_i,\]
    and
    \[\psi_\infty: \W{-1} \to T_\infty , \quad f\del \mapsto f\del + \sum_{i \geq 0} \frac{f^{(i+1)}}{(i+1)!} e_i, \]
    which induce algebra homomorphisms $\Psi_n: \Ua(\W{-1}) \to T_n$ and $\Psi_\infty: \Ua(\W{-1}) \to T_\infty$.
\end{proposition}

Notice all the homomorphisms above are well-defined because if $f \del \in \W{-1}$, then $f$ is a polynomial. Moreover, $\Psi_\infty$ is injective by Proposition 5.13 of \cite{pham2025orbit}. There is a clear relationship between the maps $\Psi_n$, $\Psi_\infty$, and the natural quotients induced by $\W{0} \rightarrow \mathfrak{g}_n \rightarrow \mathfrak{g}_{n-1}$, see the below commutative diagram.

\[
\begin{tikzcd}
 &  & T_\infty \arrow[d, "{v_n, v_{n+1}, \dots \mapsto 0}"] \\
\Ua(\W{-1}) \arrow[rru, "\Psi_\infty"] \arrow[rr, "\Psi_n"] \arrow[rrd, "\Psi_{n-1}"'] &  & T_n \arrow[d, "v_{n-1} \mapsto 0"]                    \\
&  & T_{n-1}                                              
\end{tikzcd}
\]
The homomorphisms $\Psi_n$ are particularly important for the representation theory of $\W{-1}$, as they are the crucial ingredient that lifts the orbit method from the finite-dimensional solvable Lie algebras $\g_n$ to $\W{-1}$, see \cite[Theorem 1.4]{pham2025orbit}. Therefore we call these homomorphisms the \emph{orbit homomorphisms}. This article is devoted to studying their kernels and images.

\subsection{Filtered rings}

Throughout this article, we will need to consider different graded and filtered ring structures on $\Ua(W_{\geq -1})$, $T_n$, and related rings, which we detail here.

\begin{definition} \label{Witt filtration def}
    For any Lie algebra $\g$, the universal enveloping algebra $\Ua(\g)$ has an ascending ring filtration,
    \[ \Ua_m(\g) = \sum_{j=0}^m \g^j = \kk\{g_1 \cdots g_j \mid j \leq m,\, g_1, \dots, g_j \in \mf{g}\},\]
    which is the canonical filtration on a universal enveloping algebra. We say that an element of $\Ua_m(\g)\backslash \Ua_{m-1}(\g)$ has \emph{order} $m$.
\end{definition}

\begin{definition} \label{S Witt filtration def}
We define the symmetric algebra $\Sa(\g)$ to be the associated graded ring $\gr (\Ua(\g))$ of $\Ua(\g)$, with graded pieces
\[ \Sa^m(\g) = \faktor{\Ua_m(\g)}{\Ua_{m-1}(\g)}. \]
We say that an element of $\Sa^m(\g)$ has \emph{order} $m$.
\end{definition}

We caution the reader that elements of order $m$ are described as elements of \emph{degree} $m$ in \cite{conley2007,conley2025annihilators}. In this article, \emph{degree} will refer to a particular graded structure, see \Cref{Graded rings}.

For all $i \geq -1$, we denote $\bar{e}_i = e_i + \Ua_{0}(W_{\geq -1}) \in \Sa^1(W_{\geq -1})$. It is the content of the PBW theorem that $\Sa(\W{-1}) = \kk[\bar{e}_{-1}, \bar{e}_0, \dots]$ is a commutative polynomial ring.

\begin{definition} \label{symmetrizer definition}
The symmetrizer map is $\sym\!: \Sa(\g) \to \Ua(\g)$, the $\kk$-linear map defined by
\[ \sym(g_1\cdots g_m) = \frac{1}{m!}\sum_{\sigma \in S_m}g_{\sigma(1)}\cdots g_{\sigma(m)}\]
for all $g_1, \dots, g_m \in \mf{g}$.
\end{definition}
We introduce the symmetrizer map purely because it is a right inverse of the associated graded map. That is, $\gr (\sym (x)) = x$ for all $x \in \Sa(\g)$, where $\gr: \Ua(\mf{g}) \rightarrow \Sa(\mf{g})$ is the canonical map sending $y \mapsto y + \Ua_{m-1}(\mf{g})$ if $y$ is of order $m$. 

We will also need a corresponding filtration on $T_\infty$ and $T_n$.

\begin{definition}\label{T filtration def}
    The algebra $T_\infty$ has an ascending ring filtration with $m$th filtered piece given by
    \[ \kk[t]\{\del^{m_\del}e_0^{m_0}\cdots e_{N}^{m_N} \mid m_\del + m_0 + \dots + m_N \leq m\},\]
    for all $m \geq 0$. Furthermore we give $T_n$ the filtration induced from $T_\infty$. 
\end{definition}

Of course, the filtrations on $T_n$, $T_\infty$ give associated graded rings $\gr T_n$, $\gr T_\infty$, and we denote $\gr(t), \gr(\del), \gr(v_j)$ by $\bar{t}, \bar{\del}, \bar{v_j}$ respectively. It is straightforward to see that $\gr T_\infty$, $\gr T_n$ are commutative and
\[\gr T_n = \kk[\bar{t}, \bar{\del}] \otimes \Sa(\g_n), \quad \gr T_\infty = \kk[\bar{t}, \bar{\del}] \otimes \Sa(\W{0}). \]
As previously, we use the terminology of \emph{order} to refer to the filtration on $T_n$, $T_\infty$, and to the induced grading on $\gr T_n$, $\gr T_\infty$.
The algebra homomorphisms $\Psi_n$, $\Psi_\infty$ respect the filtrations.

\begin{lemma}[{\cite[Theorem 5.7]{pham2025orbit}}] \label{filtered algebra hom}
    $\Psi_n$ and $\Psi_\infty$ are filtered algebra homomorphisms, and their associated graded maps are respectively given by
    \begin{align*}
        \Phi_n: \Sa(\W{-1}) \to \kk[\bar{t}, \bar{\del}] \otimes \Sa(\g_n), \quad f\del \mapsto f\bar{\del} + \sum_{i=0}^{n-1} \frac{f^{(i+1)}}{(i+1)!} \bar{v}_i,\\
        \Phi_\infty: \Sa(\W{-1}) \to \kk[\bar{t}, \bar{\del}] \otimes \Sa(\W{0}), \quad f\del \mapsto f\bar{\del} + \sum_{i \geq 0} \frac{f^{(i+1)}}{(i+1)!} \bar{e}_i.\\
    \end{align*}
\end{lemma}

We obtain the below commutative diagram of graded homomorphisms between commutative algebras.

\[
\begin{tikzcd}
&  & \gr T_\infty \arrow[d, "{\bar{v}_n, \bar{v}_{n+1}, \dots \mapsto 0}"] \\
\Sa(\W{-1}) \arrow[rru, "\Phi_\infty"] \arrow[rr, "\Phi_n"] \arrow[rrd, "\Phi_{n-1}"] &  & \gr T_n \arrow[d, "\bar{v}_{n-1} \mapsto 0"]  \\ &  & \gr T_{n-1}
\end{tikzcd}
\]

\subsection{Graded rings} \label{Graded rings}
The Witt algebra is a $\mathbb{Z}$-graded Lie algebra with graded pieces $\kk e_i$ for $i \geq -1$. This induces a $\mathbb{Z}$-grading on $\UW$ and the related algebras above, as follows.

\begin{definition}\label{def: degree U and S}
    $\UW$ is a $\mathbb{Z}$-graded algebra, with $m$th graded piece
    \[\UW_m= \kk\{e_{m_1}\cdots e_{m_l} \mid m_1 + \dots + m_l = m\},\]
    for all $m \in \mathbb{Z}$. Furthermore, $\Sa(\W{-1})$ is a $\mathbb{Z}$-graded algebra, with $m$th graded piece
    \[ \Sa(\W{-1})_m = \kk\{\bar{e}_{m_1}\cdots \bar{e}_{m_l} \mid m_1 + \dots + m_l = m\},\]
    for all $m \in \mathbb{Z}$.
\end{definition}

We also define a grading on the target algebras $T_n$, $T_\infty$.

\begin{definition}\label{T grading def}
    The algebra $T_\infty$ is $\mathbb{Z}$-graded with $m$th graded piece
    \[ (T_\infty)_m = \kk\{t^{m_t}\del^{m_\del}e_0^{m_0}\cdots e_{N}^{m_N} \mid N \geq 1,\, m_t - m_\del + m_1 + 2m_2 + \dots + Nm_N=m\},\]
    for all $m \in \mathbb{Z}$. Furthermore we give $T_n$ the induced grading from $T_\infty$.
\end{definition}

This grading moreover descends to the associated graded rings $\gr T_\infty, \gr T_n$, defined with respect to the order filtration of \Cref{T filtration def}.

\begin{definition} \label{gr T grading def}
    The algebras $\gr T_\infty$, $\gr T_n$ are $\mathbb{Z}$-graded with $m$th graded pieces
    \[ (\gr T_\infty)_m = \kk\{\bar{t}^{m_t}\bar{\del}^{m_\del}\bar{e}_0^{m_0}\cdots \bar{e}_{N}^{m_N} \mid N \geq 1,\, m_t - m_\del + m_1 + 2m_2 + \dots + Nm_N=m\}, \]
    and similarly 
    \[
    (\gr T_n)_m = \kk\{\bar{t}^{m_t}\bar{\del}^{m_\del}\bar{e}_0^{m_0}\cdots \bar{e}_{n-1}^{m_{n-1}} \mid m_t - m_\del + m_1 + 2m_2 + \dots + (n-1)m_{n-1}=m.\}
    \]
\end{definition}

We say that an element in the $m$th graded piece of $\UW$, $\Sa(\W{-1})$, $T_n$, $T_\infty$, $\gr T_n$, or $\gr T_\infty$, as defined in \Cref{def: degree U and S,T grading def,gr T grading def}, is of \emph{degree} $m$. The term \emph{homogeneous} will be used to refer to these elements only. In particular, $t$, $\del$ and $v_j$ are homogeneous elements of degree 1, -1, and $j$, respectively. We alert the reader that in \cite{conley2025annihilators}, this is called weight (with respect to $e_0$).\\

We caution the reader that we have now given two distinct gradings on the commutative algebras $\Sa(\W{-1}) = \gr \UW$, $\gr T_n$, and $\gr T_\infty$. To avoid confusion, we will occasionally need the convention that $\gr_j A$ is the $j$th graded piece with respect to the order grading, whilst $(\gr A)_j$ is the $j$th graded piece with the degree grading, for $A$ any vector subspace of $\UW$, $T_n$, or $T_\infty$.

By letting subspaces and quotients inherit filtrations or gradings in the usual way, the following lemma can be formally checked.
\begin{lemma} \label{gradings commute}
Let $A$ be a subspace of $\UW$, $T_n$, or $T_\infty$. Then for any $i,j \in \mathbb{Z}$,
\[ (\gr_i A)_j = \gr_i A_j. \]
That is, $(\gr A)_j = \gr A_j$.
\qed
\end{lemma}

The orbit homomorphisms respect the degree grading.

\begin{proposition}[{\cite[Theorem 5.7]{pham2025orbit}}] \label{Psi graded map prop}
    $\Psi_n$ and $\Psi_\infty$ are graded algebra homomorphisms with respect to the degree grading.
\end{proposition}

Moreover, the corresponding maps on the associated graded respect both gradings defined in this section.

\begin{proposition}[{\cite[Corollary 5.6 and Theorem 5.7]{pham2025orbit}}] \label{prop:phi map} 
    The algebra homomorphisms $\Phi_n$ and $\Phi_\infty$ are graded algebra homomorphisms with respect to both the order grading and the degree grading.
\end{proposition}

\subsection{The differentiator elements}

We now define some special elements of $\UW$, which will be useful in our computation. These elements were previously studied in \cite{billig2016classification}, and appeared in \cite{liu2015class}.
\begin{definition}\label{def:oo}
    For $m \geq 0, s \geq -1$ and $k \geq m-1$, we recursively define the \emph{differentiator} $\oo{m}{k}{s} \in \UW$ via
    \[ \oo{0}{k}{s} = e_k e_s, \quad \oo{m+1}{k}{s} = \oo{m}{k}{s} - \oo{m}{k-1}{s+1}. 
    \]
\end{definition}
The terminology of differentiator reflects that the element $\oo{m}{k}{s}$ can be viewed as the $m$th difference derivative of $\oo{0}{k}{s}$. It is straightforward from the definition that differentiators have a closed form,
\begin{align} \label{Omega long form eqn}
    \oo{m}{k}{s} = \sum_{i=0}^{m} (-1)^i \binom{m}{i} e_{k-i} e_{s+i}.
\end{align}

The differentiators are of particular interest for us, because in this article we show that they generate the kernels of orbit homomorphisms.

The vector subspace of $\UW$ spanned by all differentiators is easily seen to be closed under the adjoint action of $\W{-1}$. Thus the commutators $[e_j, \oo{m}{k}{s}]$ can be expressed as linear combinations of differentiators. We give the formulae for the cases of most use to us in this article.

\begin{lemma} \label{differentiator commutators lemma}
    For any integers $m \geq 0, s \geq -1$ and $k \geq m-1$, we have
    \begin{align}
        [e_{-1}, \oo{m}{k}{s}] &= (k+1-m)\oo{m}{k-1}{s} + (s+1) \oo{m}{k}{s-1}, \label{commutator with e-1}\\
        [e_0, \oo{m}{k}{s}] &= (k+s) \oo{m}{k}{s}, \label{commutator with e0}\\
        [e_1, \oo{m}{k}{s}] &= (k-1)\oo{m}{k+1}{s} + (s-1+m) \oo{m}{k}{s+1}, \label{commutator with e1}\\ 
        [e_2, \oo{m}{k}{s}] &= (k-2)\oo{m}{k+2}{s} + m \oo{m}{k+1}{s+1} + (s-2+m)\oo{m}{k}{s+2}.\label{commutator with e2}
     \end{align}
\end{lemma}
\begin{proof}
    We first show (\ref{commutator with e-1}) by induction on $m$. If $m=0$, then
    \[ [e_{-1}, e_{k}e_{s}] =(k+1) e_{k-1}e_{s} + (s+1)e_{k}e_{s+1},\]
    as required. Assume the induction hypothesis (\ref{commutator with e-1}) for all $k, s$. Then
    \begin{align*}
        [e_{-1}, \oo{m+1}{k}{s}] &= [e_{-1}, \oo{m}{k}{s} - \oo{m}{k-1}{s+1}] \\ 
        &= (k+1-m)\oo{m}{k-1}{s} + (s+1) \oo{m}{k}{s-1} - (k-m)\oo{m}{k-2}{s+1} - (s+2) \oo{m}{k-1}{s} \\ 
        &= (k-m)(\oo{m}{k-1}{s} - \oo{m}{k-2}{s+1}) + (s+1)(\oo{m}{k}{s-1} -  \oo{m}{k-1}{s}) \\ 
        &= (k-m)\oo{m+1}{k-1}{s}  + (s+1)\oo{m+1}{k}{s-1},
    \end{align*}
    as required.

    Next, the equation (\ref{commutator with e0}) follows as $\oo{m}{k}{s}$ is a homogeneous element of degree $k+s$.

    We now show (\ref{commutator with e1}) by induction on $m$. If $m = 0$, then
    \[ [e_1, e_{k}e_{s}] =(k-1) e_{k+1}e_{s} + (s-1)e_{k}e_{s+1}.\]
    Assume the induction hypothesis (\ref{commutator with e1}) for all $k, s$. Then
    \begin{align*}
        [e_1, \oo{m+1}{k}{s}] &= [e_1, \oo{m}{k}{s} - \oo{m}{k-1}{s+1}] \\ 
        &= (k-1)\oo{m}{k+1}{s} + (s-1+m) \oo{m}{k}{s+1} - (k-2)\oo{m}{k}{s+1} - (s+m) \oo{m}{k-1}{s+2} \\ 
        &= (k-1)(\oo{m}{k+1}{s} - \oo{m}{k}{s+1}) + (s+m)(\oo{m}{k}{s+1} -  \oo{m}{k-1}{s+2}) \\ 
        &= (k-1)\oo{m+1}{k+1}{s}  + (s+m)\oo{m+1}{k}{s+1}.
    \end{align*}

    Similarly, we show equation (\ref{commutator with e2}) by induction on $m$. If $m = 0$, then
    \[ [e_2, e_{k}e_{s}] =(k-2) e_{k+2}e_{s} + (s-2)e_{k}e_{s+2}.\]
    Assume the induction hypothesis (\ref{commutator with e2}) for all $k, s$. Then
    \begin{align*}
        [e_2, \oo{m+1}{k}{s}] &= [e_2, \oo{m}{k}{s} - \oo{m}{k-1}{s+1}] \\ 
        &= (k-2)\oo{m}{k+2}{s} + m \oo{m}{k+1}{s+1} + (s-2+m)\oo{m}{k}{s+2} \\ 
        &- (k-3)\oo{m}{k+1}{s+1} - m \oo{m}{k}{s+2} - (s-1+m)\oo{m}{k-1}{s+3} \\ 
        &= (k-2)(\oo{m}{k+2}{s} - \oo{m}{k+1}{s+1}) + (m+1)(\oo{m}{k+1}{s+1} -  \oo{m}{k}{s+2}) \\ 
        &+ (s-1+m)(\oo{m}{k}{s+2} -  \oo{m}{k-1}{s+3})  \\ 
        &= (k-2)\oo{m+1}{k+2}{s} + (m+1)\oo{m+1}{k+1}{s+1}+ (s-1+m)\oo{m+1}{k}{s+2}.
    \end{align*}
\end{proof}

%% file: image.tex
\section{The image of an orbit homomorphism}\label{sec:image}

For all $n \geq 0$, let $B_n = \Psi_n ( \Ua(\W{-1}))$ be the image of $\Psi_n$. In this section, we first characterize the image of the differentiators $\oo{2n}{2n-1}{-1}$ defined in \Cref{def:oo}.
\begin{theorem} \label{image of Omega thm}
Let $n \geq 2$. Then
    \begin{align*}
        \Psi_n (\oo{2n}{2n-1}{-1}) \in \kk^\times v_{n-1}^2.
    \end{align*}
\end{theorem}
This calculation shows that $\oo{2n}{2n-1}{-1} \in \ker \Psi_{n-1}$, so we have an immediate corollary (which we improve to an equality in \Cref{sec:kernel}).

\begin{corollary} \label{Omega in kernel cor}
Let $n \geq 1$. Then $(\oo{2n+2}{2n+1}{-1}) \subseteq \ker \Psi_{n}$. 
\end{corollary}

We then show that $B_n$ contains a large ideal of $T_n$, and hence $B_n$ is birational to $T_n$.

\begin{theorem}\label{Bn large ideal thm}
Let $n \geq 2$. Then $T_nv_{n-1}^2 \ideal B_n$ is a principal two-sided ideal of $B_n$.
\end{theorem}
The ideal $T_n v_{n-1}^2$ of $B_n$ proves to be important as it is not finitely generated as a left or right ideal of $B_n$, but is a principal two-sided ideal of $B_n$. Moreover, it plays an important role in the proof of a surprising fact: the subring $(B_n)_0$ of degree zero elements (with respect to the degree grading in \Cref{T grading def}) is Noetherian. 

\degreezeroNoetherian*

\Cref{Main result degree zero Noetherian} gives us evidence for the following conjecture.

\begin{restatable}{conjecture}{Wittdegreezero} \label{Witt degree zero conjecture}
    The subring of degree zero elements $(\Ua(\W{-1}))_0$ is a Noetherian ring.
\end{restatable}

We warmly thank Olivier Mathieu and Susan J. Sierra for bringing this interesting question to our attention. 

\subsection{The step elements} \label{The step elements}
In this subsection, we define a special operation defined on $T_\infty$ motivated by the \emph{step algebra element} $S$ in \cite{conley2025annihilators}.
Step algebras were first introduced by Mickelsson \cite{Mickelsson1973} and van den Hombergh in \cite{vdHombergh1975}, and explicit formulae for their elements were shown in \cite{Zhelobenko1990}. 
We are, in particular, interested in the step element $S$ of degree 2,
\[
    S = 4 e_2 e_0 + 2e_2 - 3 e_1^2, 
\]
which corresponds to the projective subalgebra $\mf{sl}_2 = \kk \{e_{-1}, e_{0}, e_{1}\}$ of $\W{-1}$.
\begin{definition}
    We define the linear operator $S: \Ua(\W{-1}) \rightarrow \Ua(\W{-1})$ by
    \begin{align*}
        S(x) = 4[e_2,  [e_0, x]] + 2[e_2, x] - 3[e_1, [e_1, x]].
    \end{align*}
 \end{definition}

We now explain the relationship between the operator $S$ and the differentiators $\oo{m}{k}{s}$.

 \begin{lemma}[{\cite[Section 12]{conley2025annihilators}}] \label{S Omega lemma}
     For all $n \geq 0$, $S(\oo{2n}{2n-1}{-1}) \in \kk^\times \oo{2n+2}{2n+1}{-1}$.
 \end{lemma}
 \begin{proof}
Using \Cref{differentiator commutators lemma}, parts (\ref{commutator with e0}), (\ref{commutator with e1}) and (\ref{commutator with e2}), we compute that
    \begin{align*}
        S(\oo{m}{m-1}{-1}) &= (4(m-2) +2 )[e_2, \oo{m+1}{m-1}{-1}] -3 [e_1, (m-2)\oo{m}{m}{-1} + (m-2) \oo{m}{m-1}{0}]\\
        &= (m^2-9m+12) \oo{m}{m+1}{-1} - 2 (m^2-9m+12)\oo{m}{m}{0} + (m^2-9m+12) \oo{m}{m-1}{1}\\ 
        &= (m^2-9m+12) \oo{m+2}{m+1}{-1}.
    \end{align*}
    It is straightforward to check that $m^2 -9m +12=0$ has no non-negative integer solutions. Since $\kk$ is of characteristic zero, the result follows by putting $m=2n$.
 \end{proof}

Recall the ring $T_\infty = \kk[t, \del] \otimes \Ua(\W{0})$ and the algebra embedding $\Psi_\infty: \Ua(\W{-1}) \to T_\infty$ from \Cref{def:ringT} and \Cref{Psi formula}. We now define a step operator on $T_\infty$.  
 \begin{definition}
    Let $\PsiS$ be the linear map $T_{\infty} \rightarrow T_{\infty}$ defined by
    \begin{align*}
        \PsiS(x) = 4[\Psi_\infty(e_2),  [\Psi_\infty(e_0), x]] + 2[\Psi_\infty(e_0), x] - 3[\Psi_\infty(e_1), [\Psi_\infty(e_1), x]].
    \end{align*}
\end{definition}

We can slightly simplify the form of $\PsiS$ on homogeneous elements due to the following.

\begin{lemma} \label{commute with c0 lemma}
    Let $x \in T_{\infty}$ be an element of degree $d$. Then $[\Psi_{\infty}(e_0), x] = dx$.
\end{lemma}
\begin{proof}
    Since $\Psi_{\infty}(e_0) = t\del + e_0$, the result follows from the computations
    \[ [t\del + e_0, t] = [t\del, t]=t, \quad [t\del + e_0, \del] = [t\del, \del]=-\del, \quad [t\del + e_0, e_i] = [e_0, e_i] = ie_i. \]
\end{proof}

The map $\PsiS$ respects the degree grading and the order filtration on $T_\infty$ defined in \Cref{Preliminaries}.

\begin{lemma} \label{properties of S lemma}
The linear map $\PsiS$ is a a filtered $\kk$-linear map and a graded $\kk$-linear map of degree 2. That is, if $x \in T_\infty$ is of order $r$, then $\PsiS(x)$ is of order at most $r$, and if $x$ is homogeneous of degree $d$, $\PsiS(x)$ is homogeneous of degree $d+2$.
\end{lemma}
\begin{proof}
    For all $j \geq -1$, $\Psi_{\infty}(e_j)$ is an element of $T_\infty$ of degree $j$, and hence by the properties of graded rings, if $x$ is homogeneous of degree $d$, then $[\Psi_\infty(e_j), x]$ is an element of degree $d+j$. It follows that $\PsiS$ is a graded map of degree 2.
    
    To show that $\PsiS$ is a filtered map, it is enough to check that if $x$ is of order $r$, then $[\Psi_\infty(e_j),x]$ is of order at most $r$. This follows from the fact that $\gr T_\infty$ is commutative and $\Psi_\infty(e_j)$ is of order $1$.
\end{proof}

Notice that $\PsiS$ is not a ring homomorphism -- in particular, $\PsiS(1)=0$. However, $\PsiS$ preserves a large subring of $T_{\infty}$, which we call $L$.

\begin{proposition} \label{PsiS preserve prop}
    Let $L = \kk[\del] \otimes_\kk \Ua(W_{\geq 0})$. Then $\PsiS(L) \subseteq L$.
\end{proposition}
\begin{proof}
    First we establish that
    \[ L = \{x \in T_{\infty} \mid [\del, x]=0\}.\]
    To see this, notice that $T_\infty$ is a skew polynomial extension of $L$ by $t$, and hence
    \[ T_{\infty} = \bigoplus_{n \geq 0} L t^n.\]
    Further, for all $r \in L$,
    \[ [\del, rt^n] = r [\del, t^n] = nrt^{n-1}.\]
    Thus if
    \[ z = \sum_{n \geq 0} r_nt^n \in T_{\infty},\]
    we can compute that
    \[ [\del, z] = \sum_{n \geq 1} nr_nt^{n-1}.\]
    Since $\kk$ is of characteristic zero, it follows that $z$ commutes with $\del$ if and only if $r_n=0$ for all $n \geq 1$, hence we have the claimed equality. 
    
    Now, let $x \in L$. Because $\PsiS$ is $\kk$-linear, we may assume without loss of generality that $x$ is a homogeneous element of degree $d$.
    
    For ease, write $c_i = \Psi_\infty(e_i)$, and note that $\del=c_{-1}$, and $[c_i, {c}_j] = (j-i){c}_{i+j}$. Also by \Cref{commute with c0 lemma}, $[{c}_0, x] = dx$ for any homogeneous element $x$ of degree $d$. It follows that
    \[ \PsiS(x) = (4d+2)[{c}_2, x] - 3[{c}_1, [{c}_1, x]].\]
    It follows from multiple applications of the Jacobi identity that
    \[ [\del, [c_2, x]] = 3[c_1, x], \quad [\del, [c_1, [c_1,x]]] = (4d+2)[c_1,x],\]
    and hence $[\del, \PsiS(x)]=0$, meaning that $\PsiS(x) \in L$.
\end{proof}

We thank Charles Conley for the following observation. The space of lowest weight vectors of $T_\infty$ as a representation of $\mf{sl}_2 = \kk\{\del, \Psi_\infty(e_0), \Psi_\infty(e_1)\}$ is precisely the homogeneous elements of $L = \{ x \in T_{\infty} \mid [\del, x]=0\}$, with the weight equal to the degree of the element. Thus \Cref{PsiS preserve prop} shows that the ``step algebra element'' $\PsiS$ sends a lowest weight vector of $T_\infty$ to a lowest weight vector of $T_\infty$.

We can now prove the following about the images of the differentiators.

\begin{proposition} \label{image of Omega weak thm}
    Let $n \geq 2$. Then
    \begin{align*}
    \Psi_n (\oo{2n}{2n-1}{-1}) \in \kk v_{n-1}^2.
    \end{align*}
\end{proposition}
\begin{proof}
    Clearly 
    \[ \Psi_\infty(\oo{0}{-1}{-1}) = \del^2 \in L. \]
    Now, $\Psi_\infty(S(x)) = \PsiS(\Psi_{\infty}(x))$ for all $x \in U(W_{\geq -1})$. Thus an easy induction using \Cref{S Omega lemma} and \Cref{PsiS preserve prop} implies that
    \[ \Psi_\infty(\oo{2n}{2n-1}{-1}) \in L \]
    for all $n \geq 0$.
    
    Next let $n \geq 2$, and notice that $\oo{2n}{2n-1}{-1}$ is a homogeneous element of degree $2n-2$ and order at most $2$, hence so is $\Psi_\infty(\oo{2n}{2n-1}{-1})$, by \Cref{Psi graded map prop}. The $\kk$-vector space of homogeneous elements of $L$ of degree $2n-2$ and order at most $2$ has a basis
    \[ \{e_{i}e_{2n-2-i} \mid 0 \leq i \leq 2n-2\} \cup \{\del e_{2n-1}, e_{2n-2}\}.\]
    $T_n$ is a quotient of $T_\infty$ by the two-sided ideal generated by $\{e_m \mid m \geq n\}$, and the only basis element not contained in this ideal is $e_{n-1}^2$. Therefore $\Psi_n(\oo{2n}{2n-1}{-1}) \in \kk v_{n-1}^2$, as required.
\end{proof}

Notice that \Cref{Omega in kernel cor} follows immediately from \Cref{image of Omega weak thm}, despite being weaker than \Cref{image of Omega thm}.

To complete \Cref{image of Omega thm}, we only need to show that $\Psi_n(\oo{2n}{2n-1}{-1}) \neq 0$ as follows.

\begin{proof}[Proof of \Cref{image of Omega thm}]
Recall that $T_\infty$ is a skew polynomial ring over $L$, and hence is a free right $L$-module on the basis $\{1,t,t^2, \dots\}$. Let $\Pi:T_\infty \rightarrow L$ be the $L$-module homomorphism
\[ \Pi \Big( \sum_{n \geq 0} t^nr_n \Big) = r_0.\]
Let $i:L \rightarrow T_\infty$ be the natural inclusion and $x \in T_\infty$. Then 
\begin{align*}
    x \in L \Leftrightarrow i \circ \Pi(x)=x.
\end{align*}
Let $X = \Psi_\infty(\oo{2n}{2n-1}{-1})$. By \Cref{PsiS preserve prop}, $i \circ \Pi(X)=X$. Now, by (\ref{Omega long form eqn}), we have that
\[ \Psi_\infty(\oo{2n}{2n-1}{-1}) = \sum_{j=0}^{2n} (-1)^j \binom{2n}{j} \Psi_{\infty}(e_{2n-1-j})\Psi_{\infty}( e_{j-1}).\]
By considering the monomials in the formula for $\Psi_{\infty}(e_j)$ and their products, we can compute that if $j \in \{1, \dots, 2n-1\}$,
\[ \Pi(\Psi_{\infty}(e_{2n-1-j})\Psi_{\infty}( e_{j-1})) = e_{2n-1-j}e_{j-1},\]
and
\[ \Pi(\Psi_{\infty}(e_{2n-1})\Psi_{\infty}( e_{-1})) = e_{2n-1} \del = \del e_{2n-1}, \quad \Pi(\Psi_{\infty}(e_{-1})\Psi_{\infty}( e_{2n-1})) = \del e_{2n-1} + 2n e_{2n-2} . \]
It follows that
\begin{align*}\label{eq:inftyomega}
    X = i \circ \Pi(X) = \sum_{j=1}^{2n-1} (-1)^j\binom{2n}{j}e_{2n-1-j}e_{j-1} + 2 \del e_{2n-1} + 2n e_{2n-2}.
\end{align*}
Thus, where $p_n: T_\infty \rightarrow T_n$ is the quotient map sending $e_j \mapsto v_j$ if $j \leq n-1$ and zero otherwise,
\[ \Psi_n(\oo{2n}{2n-1}{-1}) = p_n(X) = (-1)^n\binom{2n}{n}v_{n-1}^2 \in \kk^\times v_{n-1}^2, \]
as required (note $\kk$ is of characteristic zero).
\end{proof}

\subsection{Birationality} \label{Birationality}

\Cref{image of Omega thm} shows that $B_n$ contains $v_{n-1}^2$ if $n \geq$ 2. In this subsection, we prove \Cref{Main result large image}, showing that $B_n$ in fact contains the two-sided ideal of $T_n$ generated by $v_{n-1}^2$. (Note that $v_{n-1}$ is a normal element of $T_n$.) From this we deduce that $B_n$ and $T_n$ are birational algebras, in the sense that their rings of quotients are equal.

In fact, we show a slightly stronger result.

\begin{proposition} \label{two-sided ideal image prop}
    Let $n \geq 2$ and $m \geq 1$ be such that $v_{n-1}^m \in B_n$. The two-sided ideal of $B_n$ generated by $v_{n-1}^m$ is
    \[ B_nv_{n-1}^mB_n = T_nv_{n-1}^m \subseteq B_n.\]
\end{proposition}

Because $v_{n-1}$ is a normal element, $T_nv_{n-1}^m = v_{n-1}^mT_n$ is a two-sided ideal of $T_n$. To show \Cref{two-sided ideal image prop}, we prove \Cref{t element lemma,T_n generation lemma}.

\begin{lemma} \label{t element lemma}
    Let $n \geq 2$ and $m \geq 1$ be such that $v_{n-1}^m \in B_n$. We have $\kk[t]v_{n-1}^m \subseteq B_nv_{n-1}^mB_n$.
\end{lemma}
\begin{proof}
For ease, let $f=v_{n-1}^m$ and $I = B_nv_{n-1}^mB_n$. Note $f$ commutes with the elements $t, \del$, and $v_1, \dots, v_{n-1} \in \mathfrak{g}_n$. Thus we can compute that for any $i \geq -1$,
\begin{align*} [\Psi_n(e_i), f] &= [t^{i+1}\partial, f] + \sum_{j=0}^{n-1} \begin{pmatrix} i+1\\ j+1 \end{pmatrix}[t^{i-j}v_j, f]\\ &= (i+1)[t^iv_0, f]\\ &=(i+1)t^i[v_0,f]\in I. \end{align*}
So if $i \geq 0$, it follows that $t^i[v_0,f] = (n-1)mt^if \in I$. Since $m \geq 1$ and $n \geq 2$, the result follows by taking $\kk$-linear combinations.
\end{proof}

\begin{lemma} \label{T_n generation lemma}
    For any $n \geq 0$, $T_n$ is generated as a left or right $B_n$-module by $\{1,t,t^2, \dots\}$.
\end{lemma}
\begin{proof}
    We show the corresponding statement for $T_\infty$, and the lemma follows by taking homomorphic images.
    
    First notice that for any $m \geq -1$, $\Psi_\infty(e_m)t = t\Psi_\infty(e_m) +t^{m+1}$, from which it follows that the subalgebra of $T_\infty$ generated by $B_\infty$ and $t$ is equal to
    \[ A = B_\infty\{1,t,t^2, \dots\} = \{1,t,t^2, \dots\}B_\infty.\]
    Thus, we show $A=T_\infty$. Clearly $t, \del \in A$. If $e_0, \dots, e_{m-1} \in A$, then the defining formula of \Cref{Psi formula} implies
    \[ \Psi_\infty(e_m) - e_m \in A,\]
    and hence $e_m \in A$. It follows by induction that $A=T_\infty$, as required.
\end{proof}

\Cref{two-sided ideal image prop} follows immediately from \Cref{t element lemma,T_n generation lemma}. We thus deduce \Cref{Bn large ideal thm}.

\begin{proof}[Proof of \Cref{Bn large ideal thm}]
    By \Cref{image of Omega thm}, ${v_{n-1}^2  \in \kk\Psi_n (\oo{2n}{2n+1}{-1}) \subseteq B_n}$. Hence $v_{n-1}^2$ satisfies the hypothesis of \Cref{two-sided ideal image prop}, therefore
    \[ B_nv_{n-1}^2B_n = T_nv_{n-1}^2 \subseteq B_n,\]
    is a principal two-sided ideal of $B_n$, as required.
\end{proof}

Despite the ideal $T_nv_{n-1}^2$ is generated by a single element as a two-sided ideal, we now demonstrate that it cannot be finitely-generated as a left or right ideal of $B_n$.

\begin{theorem} \label{non fg image ideal}
    Let $n \geq 2$. The ideal $T_nv_{n-1}^2 \subseteq B_n$ is not finitely-generated as a left ideal or as a right ideal.
\end{theorem}
\begin{proof}
Because $T_n$ is a domain, it is straightforward that $T_nv_{n-1}^2 = v_{n-1}^2T_n \cong T_n$ as a left or right $B_n$-module. Therefore, it is sufficient to show that $T_n$ is not a finitely-generated left or right $B_n$-module. By \Cref{T_n generation lemma},
\[ T_n = B_n\{1,t,t^2, \dots\} = \{1,t,t^2, \dots\}B_n,\]
and hence it is enough to show that for any $N \geq 0$, $T_n$ is not generated by $\{1, \dots, t^N\}$ as a left $B_n$-module or a right $B_n$-module.

For the left module structure, suppose $T_n$ is generated by $\{1, \dots, t^N\}$ as a left $B_n$-module. Then $T_1$ is generated by $\{1, \dots, t^N\}$ as a left $B_1$-module. The following argument is essentially given in \cite[Proposition 2.5.6]{buzaglo2024lie}.  Consider the right ideal of $T_1$ given by $J = \del T_1 + (v_0-1)T_1$. Then
\[ \Psi_1(e_j) = \del t^{j+1} + (j+1)(v_0-1)t^j \in J\]
for all $j \geq -1$. So it follows that $B_1 \subseteq \kk + J$. It follows that the left $B_1$-submodule of $T_1$ generated by $\{1, \dots, t^N\}$ is contained in 
\[ \sum_{j=0}^N (\kk + J)t^j = \sum_{j=0}^N \kk t^j + J.\]
But, $t^{N+1} \not \in J$, and it then follows by taking homogeneous components that $t^{N+1}$ is not contained in the above. Hence $T_1$ is not generated by $\{1, \dots, t^N\}$ as a left $B_1$-module, giving the required contradiction.

For the right module structure, suppose $T_n$ is generated by $\{1, \dots, t^N\}$ as a right $B_n$-module. Then $T_0 = A_1 = k[t, \del]$ is generated by $\{1, \dots, t^N\}$ as a right $B_0$-module. It is straightforward to see that $B_0 = \kk + \kk[t,\del]\del$, and then
\[ \sum_{j=0}^N t^j(\kk + \kk[t,\del]\del) = \sum_{j=0}^N \kk t^j + \kk[t,\del]\del.\]
Clearly $t^{N+1}$ is not contained in the above, giving a contradiction.
\end{proof}

Combining \Cref{Bn large ideal thm} and \Cref{non fg image ideal} thus gives the first main result of this article.

\largeimage*

Similarly to $T_nv_{n-1}^2$, we expect that every two-sided ideal of $B_n$ has a finite generating set.

\begin{conjecture} \label{Bn ACC conjecture}
    For any $n \geq 0$, the two-sided ideals of $B_n$ satisfy the ascending chain condition.
\end{conjecture}

Naturally, \Cref{Bn ACC conjecture} would be implied by \Cref{UW ACC conjecture}. The cases $n \leq 1$ of \Cref{Bn ACC conjecture} are essentially known, for $B_0$ is Noetherian, whilst \cite[Proposition 6.6]{sierra2016maps} proves the conjecture for $\Psi_1(\Ua(\W{1}))$.

A similar result to \Cref{Main result large image} holds for $n=1$, originally due to \cite{conley2007}.

\begin{lemma} \label{B_1 large image lemma}
    The image $B_1$ contains $T_1(v_0^2-v_0)$, which is not finitely-generated as a left or right ideal of $B_1$.
\end{lemma}
\begin{proof}
    We can compute that $\Psi_1(\oo{2}{k}{-1}) \in \kk^\times t^{k-1}(v_0^2-v_0) \in B_1$ for all $k \geq 1$, therefore
    $T_1(v_0^2-v_0) = (v_0^2-v_0)T_1$ is an ideal of $B_1$ by \Cref{T_n generation lemma}. The proof of \Cref{non fg image ideal} shows that $T_1$ is not a finitely-generated left or right $B_1$-module, giving the result.
\end{proof}

We now deduce from \Cref{Main result large image} that $B_n$ and $T_n$ are birational and of the same GK-dimension.

\birationalimage*
   
\begin{proof}
    First notice that $T_n$ is a domain, and hence so is the subalgebra $B_n$, with
    \[ \GKdim B_n \leq \GKdim T_n = \GKdim A_1 + \dim \g_n = n+2,\]
    by Theorem 1 of \cite{warfield1984gk}. So $B_n$ and $T_n$ are domains of finite GK-dimension.
    Thus $B_n$ and $T_n$ are left and right Ore domains, and hence have unique left and right quotient rings $Q(T_n)$, $Q(B_n)$, by Corollary 8.1.21 and Corollary 2.1.14 of \cite{mcconnell2001noncommutative}. If $n \geq 2$, then since ${v_{n-1}^{-2} \in Q(B_n)}$ by \Cref{image of Omega thm}, we have
    \[ T_n = (T_n v_{n-1}^2) v_{n-1}^{-2}  \subseteq Q(B_n), \]
    hence $Q(T_n) \subseteq Q(B_n)$. In fact we have equality because $B_n$ is a subring of $T_n$, so the two algebras are birational. Very similar reasoning applies for $n=1$, by \Cref{B_1 large image lemma}.
    
    We now deduce that $B_n$, $T_n$ have equal GK-dimension. To do this, we recall the Gelfand-Kirillov transcendence degree $\mathrm{Tdeg}$, studied in \cite{zhang1996}. Then, \cite[Theorem 1.1]{zhang1996} implies
    \[ \mathrm{Tdeg}(\Ua(\mf{g}_n)[\del]) = \GKdim(\Ua(\mf{g}_n))+1=n+1,\]
    and next \cite[Proposition 9.3]{zhang1996} implies $\mathrm{Tdeg}(T_n) \geq n+2$. Because $\UW$ is generated by $\{e_{-1},e_0,e_1,e_2\}$ as a $\kk$-algebra, $B_n$ is also a finitely-generated $\kk$-algebra. Thus \cite[Corollary 1.2]{zhang1996} implies
    \[ \GKdim B_n \geq \mathrm{Tdeg}(T_n) \geq n+2,\]
    because $B_n$ is birational to $T_n$. Since $\GKdim T_n =n+2$, equality must hold.
    
    Finally, $B_n$ is not left or right Noetherian by \Cref{non fg image ideal} and \Cref{B_1 large image lemma}.
\end{proof}

\subsection{The subring \texorpdfstring{$(B_n)_0$}{Bn0} is Noetherian} \label{The subring Bn0 is Noetherian}
Although $\UW$ is known to be non-Noetherian, it is not known whether the subring of degree zero elements $(\UW)_0$ remains non-Noetherian. In this subsection, we answer the question for $B_n$: we show that $(B_n)_0$ is Noetherian for all $n$. Recall the definition of degree from \Cref{def: degree U and S,T grading def}.

\degreezeroNoetherian*

We first show that $(\Ua(\W{-1}))_0$ is finitely generated as an algebra.

\begin{lemma}\label{lem: UW0 generators}
    The subring of degree zero elements $(\Ua(\W{-1}))_0$ is generated as an algebra by $\{e_{-1}^k e_{k} \mid k \geq 0\}$.
\end{lemma}
\begin{proof}
    Firstly, for all $a,b \geq 0$, an easy induction on $a$ implies that 
    \begin{equation}\label{UW0 generators help}
        [e_{-1}^a, e_b] \in \kk \{e_{-1}^{k} e_{b-a+k} \mid k \leq a\}.
    \end{equation}

    Now by the PBW theorem, a $\kk$-basis of $(\UW)_0$ is
    \[\{e_{-1}^{k_0 + \dots k_{\ell}} e_{k_0} \dots e_{k_\ell} \mid \ell \in \NN, 0 \leq k_0 \leq \dots \leq k_\ell\}.\]
    Thus, it is enough to show that $e_{-1}^{k_0 + \dots k_{\ell}} e_{k_0} \dots e_{k_\ell}$ can be written as a linear combination of products of $\{e_{-1}^ke_{k} \mid k \geq 0\}$. We proceed by induction on the length $\ell$, with the base case $\ell = 1$ being trivial. Since 
    \begin{align*}
        e_{-1}^{k_0 + \dots + k_{\ell+1}} e_{k_0} \dots e_{k_{\ell+1}} = e_{-1}^{k_0} e_{k_0} e_{-1}^{k_1 + \dots + k_{\ell+1}} e_{k_1} \dots e_{k_{\ell+1}} + e_{-1}^{k_0}[e_{-1}^{k_1 + \dots + k_{\ell+1}}, e_{k_0}] e_{k_1} \dots e_{k_{\ell+1}},
    \end{align*}
    induction, using \eqref{UW0 generators help} in the second term, implies the result.
\end{proof}

\begin{lemma}\label{lem: UW0 finite generators}
    The subring of degree zero elements $(\Ua(\W{-1}))_0$ is finitely generated as an algebra by 
    \[\{e_0, e_{-1} e_{1}, e_{-1}^2 e_{2}\}.\]
\end{lemma}
\begin{proof}
    Let $A$ be the subalgebra of $(\Ua(\W{-1}))_0$ generated by $\{e_0, e_{-1} e_{1}, e_{-1}^2 e_{2}\}$. We first observe that $e_{-1}^k e_{k-1} e_{1} \in A$ for all $k \geq 1$, by an easy induction using the identity
    \[e_{-1}^{k+1} e_{k} e_{1} = e_{-1}^{k} e_{k} e_{-1} e_{1} + (k+1) e_{-1}^{k} e_{k-1} e_{1}.\]

    Now, by \Cref{lem: UW0 generators}, it is enough to show that $e_{-1}^k e_{k} \in A$ for all $k \geq 0$. We proceed by induction, with the base cases $k =0, 1, 2$ trivial. We can compute that 
    \begin{align*}
        [e_{-1}^k e_{k}, e_{-1} e_{1}] &= e_{-1} [e_{-1}^k e_{k},  e_{1}] + [e_{-1}^k e_{k}, e_{-1}] e_{1} \\ 
        &= e_{-1} [e_{-1}^k,  e_{1}] e_{k} + (1-k) e_{-1}^{k+1}  e_{k+1}  - (k+1) e_{-1}^k e_{k-1} e_{1},
    \end{align*}
    therefore by reordering the terms, and by the above, it suffices to show that $e_{-1} [e_{-1}^k,  e_{1}] e_{k} \in A$. Now, $e_{-1}[e_{-1}^k, e_1]$ is an element of $\Ua(\W{-1})$ of order $k+1$ and degree $-k$, which forces
    \[e_{-1}[e_{-1}^k, e_1] \in \kk \{e_0 e_{-1}^{k}, e_{-1}^{k}\}.\] 
    Thus by induction, $e_{-1} [e_{-1}^k,  e_{1}] e_{k} \in A$, so $A = (\Ua(\W{-1}))_0$ as required. 
\end{proof}

With \Cref{lem: UW0 finite generators}, we can now prove \Cref{Main result degree zero Noetherian} in the cases $n \leq 1$.

\begin{lemma} \label{B00 B10 Noetherian}
    The algebras $(B_0)_0$ and $(B_1)_0$ are commutative Noetherian rings.
\end{lemma}
\begin{proof}
    Note that $(T_1)_0 = \kk[t\del, v_0]$ is commutative, so $(B_0)_0, (B_1)_0$ are commutative. Since $(\UW)_0$ is a finitely-generated $\kk$-algebra, so are $(B_0)_0, (B_1)_0$. It follows from the Hilbert basis theorem that $(B_0)_0, (B_1)_0$ are Noetherian. (In fact, $B_0 = \kk \oplus A_1 \del$ so $(B_0)_0 = \kk[t\del]$ is a polynomial ring.)
\end{proof}

Thus, we assume $n \geq 2$ for the rest of the subsection. We will prove \Cref{Main result degree zero Noetherian} by showing that the associated graded ring $\gr (B_n)_0$ (defined with respect to the order filtration) is Noetherian. The advantage of this approach is that we may work almost exclusively with commutative rings for the rest of this section. We define some homomorphisms between these rings.

\begin{definition}
    Let $m \leq n$ be non-negative integers. We define the $\kk$-algebra homomorphisms
    \[ \iota_{m,n} : \gr T_m \rightarrow \gr T_n, \quad \bar{\del} \mapsto \bar{\del},\; \bar{t} \mapsto \bar{t},\; \bar{v}_j \mapsto \bar{v}_j,\]
    for all $j \in \{0, \dots, m-1\}$, and
    \[ \pi_{n,m} = \gr p_{n,m}: \gr T_n \rightarrow \gr T_m\]
    where $p_{n,m}: T_n \rightarrow T_m$ is the algebra homomorphism induced by the quotient $\g_n \rightarrow \g_m$.
\end{definition}

The homomorphisms $\iota_{m,n}$, $\pi_{n,m}$ have good properties.

\begin{lemma}\label{gr is surj}
    Let $m \leq  n$ be positive integers. Then, $\iota_{m,n}$ is injective, $\pi_{n,m}$ is surjective and respects degree, and $\pi_{n,m} \circ \iota_{m,n} = \mathrm{id}_{\gr T_m}$. Moreover,
    \[\pi_{n,m}(\gr B_n) = \gr B_m.\]
\end{lemma}
\begin{proof}
    It is straightforward to check from the definition that
    \[ \pi_{n,m}(\bar{\del})=\bar{\del},\; \pi_{n,m}(\bar{t})=\bar{t},\; \pi_{n,m}(\bar{v}_j) = \bar{v}_j,\]
    if $j \leq m-1$, whilst $\pi_{n,m}(\bar{v}_j)=0$ if $j \geq m$. Thus $\pi_{n,m}$ is surjective and respects degree, whilst $\pi_{n,m} \circ \iota_{m,n}$ is the identity on the $\kk$-algebra generators $\bar{t}, \bar{\del}, \bar{v}_0, \dots, \bar{v}_{m-1}$ of $\gr T_m$, hence must be the identity. In particular, $\iota_{m,n}$ must be injective.

    Since the order filtrations on $B_m$, $B_n$ are induced from that on $\UW$, the quotient map
    \[ p_{n,m}|_{B_n}: B_n \rightarrow B_m\]
    must be a strict map, so by \cite[Corollary D.III.7]{nastasescu1982}, it follows that $\pi_{n,m}(\gr B_n) = \gr B_m$.
\end{proof}

Because of \Cref{gr is surj}, the homomorphisms $\iota_{m,n}$ form a compatible system of injective ring homomorphisms. For this reason we identify $\gr T_m$ as a subring of $\gr T_n$, via $\iota_{m,n}$, for the remainder of this subsection.\\

We now recall an  important definition from commutative algebra. 

\begin{definition}
    An element $a$ of a commutative ring $R'$ is \emph {integral} over a subring $R$ if there exists a monic polynomial $f \in R[x]$ with $f(a) = 0$, that is,
    \[a^n +c_{n-1}a^{n-1} +\dots +c_0 = 0\]
    for some $n > 0$ and $c_0,\dots, c_{n-1} \in R$. We say that $R'$ is integral over $R$ if every element of $R'$ is integral over $R$.
\end{definition}

We can now show the central technical result needed for the proof of \Cref{Main result degree zero Noetherian}. Recall from \Cref{gradings commute} that the associated graded of the zero degree subring $\gr (B_n)_0$ coincides with $(\gr B_n)_0$, and the same for $T_n$.

\begin{theorem}\label{grTn0 integral}
    The ring $(\gr T_n)_0$ is integral over $(\gr B_n)_0$. 
\end{theorem}
\begin{proof}
For all $j \in \NN$, let $\bar{B}_j = \gr B_j$ and $\bar{T}_j = \gr T_j$, with the corresponding elements of degree zero, as defined by \Cref{gr T grading def}, denoted $(\bar{B}_n)_0 \leq (\bar{T}_n)_0$. Let $Q$ be the integral closure of $(\bar{B}_n)_0$ in $(\bar{T}_n)_0$, which is a subring of $(\bar{T}_n)_0$. We will show that $Q=(\bar{T}_n)_0$. 
    
We will prove, via strong downward induction, that
\begin{equation}\label{eq:grTn0 integral induction}
    (\bar{T}_{i} \bar{v}_{i-1})_0 \in Q \text{ for all } i = n, \dots, 2,
\end{equation}
recalling $\bar{T}_{i}$ is identified as a subring of $\bar{T}_n$ by \Cref{gr is surj}.

The base case that $(\bar{T}_{n} \bar{v}_{n-1})_0$ is integral over $(\bar{B}_n)_0$ follows by \Cref{Bn large ideal thm}, because $T_n v_{n-1}^2 \subseteq B_n$, and thus 
\[(\bar{T}_{n} \bar{v}_{n-1})_0 (\bar{T}_{n} \bar{v}_{n-1})_0 = (\bar{T}_{n} \bar{v}^2_{n-1})_0 \in (\bar{B}_n)_0.\]
So every element of $(\bar{T}_{n} \bar{v}_{n-1})_0$ satisfies a monic quadratic equation over $(\bar{B}_n)_0$.

By \Cref{gr is surj}, the homomorphism
\[ \pi_{n, i}|_{\bar{B}_n}: \bar{B}_n \to \bar{B}_{i}\]
is surjective, and by \Cref{Bn large ideal thm}, $\bar{T}_{i} \bar{v}_{i-1}^2 \subseteq \bar{B}_{i}$. Thus for all $x \in (\bar{T}_{i})_{-2(i-1)}$, there exists $y \in \ker \pi_{n, i}$ such that 
\[ x \bar{v}_{i-1}^2 + y \in \bar{B}_n, \]
and without loss of generality $y \in (\bar{T}_n)_0$. Now, $\ker \pi_{n, i}$ is the ideal of $\bar{T}_n$ generated by 
$\bar{v}_{i}, \dots, \bar{v}_{n-1}$. In fact, since $\bar{T}_n$ is commutative, we can rearrange terms to see that
\[ \ker \pi_{n, i} = \sum_{j=i+1}^{n} \bar{T}_{j} \bar{v}_{j-1}.\]
Since $y \in \ker \pi_{n, i}$, it follows that $y \in Q$ by induction. Thus 
\[x \bar{v}_{i-1}^2 = (x \bar{v}_{i-1}^2 +y)-y \in Q,\]
for all $x \in (\bar{T}_{i})_{-2(i-1)}$. So, using the same reasoning as in the base case, we obtain $(\bar{T}_{i} \bar{v}_{i-1})_0 \subseteq Q$ as desired.

Note that 
\[(\bar{T}_n)_0 = \kk[\bar{t}\bar{\del}, \bar{t}\bar{\del} + \bar{v_0}, \bar{\del} \bar{v}_1, \dots, \bar{\del} \bar{v}_{n-1}],\]
and by \eqref{eq:grTn0 integral induction}, $\bar{\del} \bar{v}_1, \dots, \bar{\del} \bar{v}_{n-1} \in Q$. Furthermore, $Q$ contains
\[\gr (\Psi_n(\oo{2}{1}{-1})) = 2(\bar{\del} \bar{v}_1 - \bar{v}_0^2),\]
thus $\bar{v}_0^2 \in Q$ and $\bar{v}_0 \in Q$. Since $\bar{t}\bar{\del} + \bar{v_0} = \gr(\Psi_n(e_{0})) \in Q$, it follows that $(\bar{T}_n)_0 \subseteq Q$, as required.
\end{proof}

A fundamental connection between integral and finite extensions is characterized by the following well-known Proposition.

\begin{proposition} \label{fg integral extension is finite}
    Let $R'$ be a commutative ring and $R$ be a subring, such that $R'$ is a finitely-generated $R$-algebra. Then, $R'$ is a finitely-generated $R$-module if and only if $R'$ is integral over $R$.
\end{proposition}
\begin{proof}
    This follows from \cite[Theorem 9.1]{Matsumura1989} and induction.
\end{proof}

Applying \Cref{fg integral extension is finite} to \Cref{grTn0 integral}, we obtain the following.

\begin{corollary}\label{grTn0 fg}
    The commutative ring $(\gr T_n)_0$ is a finitely-generated $(\gr B_n)_0$-module.
\end{corollary}
\begin{proof}
    Note that
    \[ (\gr T_n)_0 = \kk[\bar{t}\bar{\del}, \bar{t}\bar{\del} + \bar{v_0}, \bar{\del} \bar{v}_1, \dots, \bar{\del} \bar{v}_{n-1}]\]
    is a finitely-generated $\kk$-algebra, so must be a finitely-generated $(\gr B_n)_0$-algebra. Since $(\gr T_n)_0$ is integral over $(\gr B_n)_0$ by \Cref{grTn0 integral}, the result is immediate from \Cref{fg integral extension is finite}.
\end{proof}

We can now deduce that $(B_n)_0$ is Noetherian.

\begin{proof}[Proof of \Cref{Main result degree zero Noetherian}]
    Due to \Cref{B00 B10 Noetherian}, we assume $n \geq 2$. Note that $(\gr T_n)_0$ is a domain, hence is a finitely-generated faithful $(\gr B_n)_0$-module by \Cref{grTn0 fg}. Since $(\gr T_n)_0$ is a Noetherian ring, it follows from \cite[Theorem 3.6]{Matsumura1989} that $(\gr B_n)_0$ is a Noetherian ring. Moreover, $(\gr B_n)_0 = \gr (B_n)_0$ by \Cref{gradings commute}. Then \cite[Theorem 1.6.9]{mcconnell2001noncommutative} implies that $(B_n)_0$ is a (left and right) Noetherian ring.
\end{proof}

Recall that $\UW$ is not left or right Noetherian, and this can be proved by showing that the homomorphic images $B_n$ are not Noetherian for $n \geq 1$. When $n=1$, this is essentially the technique of \cite{SierraWalton}. In light of \Cref{Main result degree zero Noetherian}, we therefore make the following conjecture.

\Wittdegreezero*

We remark that $\gr (\Ua(\W{-1}))_0$ is a polynomial ring in infinitely many variables $\bar{e}_{-1}^n\bar{e}_n$, so is not Noetherian. Thus, the technique of lifting Noetherianity from an associated graded ring, as used in the proof of \Cref{Main result degree zero Noetherian}, cannot be used to prove (or indeed disprove) \Cref{Witt degree zero conjecture}.

However, the combination of \Cref{Main result degree zero Noetherian} with \Cref{Main result generic kernel} provides direct evidence for the conjecture. Namely, any ascending chain of left or right ideals of $\UW$ that eventually contains a two-sided ideal $(\Omega)_0 \ideal (\UW)_0$, for some non-zero differentiator $\Omega$ as in \Cref{Main result generic kernel}, must terminate.

%% file: kernel.tex
\section{The kernel of an orbit homomorphism is principal}\label{sec:kernel}
Throughout this section, let $n$ be a positive integer unless otherwise stated. So far, we have computed properties of the image of the orbit homomorphisms $\Psi_n$. We now improve our understanding by also computing their kernels. In fact, we show that $\ker \Psi_n$ is a principal two-sided ideal: generated by a single element of $\Ua(\W{-1})$.

\begin{theorem} \label{kernel thm}
    For all $n \geq 1$, the kernel of the map $\Psi_n$ is the principal ideal of $\Ua(\W{-1})$ generated by $\oo{2n+2}{2n+1}{-1}$. In other words,
    \begin{align*}
        \ker \Psi_n = (\oo{2n+2}{2n+1}{-1}).
    \end{align*}
\end{theorem}

The kernels of $\Psi_0$ and $\Psi_1$ were computed previously in \cite{conley2007}, where they used the result to calculate the generators of the annihilators of tensor density modules. We thus extend their result to all non-negative integer $n$.

In sub\cref{Conley-Goode method and injectivity,Basic properties of J,Representation theory and the projection of I,Proof of kernel thm} we prove \Cref{kernel thm} and deduce a number of consequences, including that $\ker \Psi_n$ is not finitely generated as a left or right ideal. As a result, the $\ker \Psi_n$ form an important class of examples of ideals of $\Ua(\W{-1})$ that are finitely generated as two-sided ideals, but not finitely generated as left or right ideals. We also obtain a fuller description of the image of $\Psi_n$, see \Cref{B_n kernel thm} and the discussion following.

We then go further by classifying all ideals generated by differentiators, in subsection \ref{Generation of ker Psi_n}. In fact, all such ideals are kernels of orbit homomorphisms.

\generickernel*

The generating set of $\ker \Psi_n$ above has the remarkable property that any non-zero element of the set generates $\ker \Psi_n$ as a two-sided ideal.

Finally in sub\cref{Primitive ideals}, we show that many of these ideals are primitive.

\subsection{The Conley-Goode method and injectivity} \label{Conley-Goode method and injectivity}

The crucial framework for proving \Cref{kernel thm} comes from \cite{conley2025annihilators}, where a method for determining annihilators of representations of Lie algebras of vector fields is described. In this article, we specialise their main result to compute kernels generated by elements of a single order. Let us briefly summarise this specialisation.

\begin{theorem}[{\cite[Proposition 3.1]{conley2025annihilators}}] \label{conley goode method thm}
Let $\g$ be a Lie algebra, $T$ an associative algebra, and $\Psi: \Ua(\g) \rightarrow T$ be a $\kk$-algebra homomorphism. Suppose that there exists a positive integer $d$, and vector subspaces $J \subseteq \g$, $\mathcal{J} \subseteq \Ua(\g)$, and $I \subseteq \ker \Psi \cap \Ua_d(\g)$, such that
\begin{enumerate}[label=(\alph*)]
    \item $\Psi$ is injective on $\mathcal{J}$, \label{injective condition}
    \item $\Ua_{d-1}(\g) \subseteq \mathcal{J}$, \label{inclusion condition}
    \item for all $m \geq d$, $\gr_m  \mathcal{J}= \Sa^{m-d+1}(J)\Sa^{d-1}(\g)$ as subsets of $ \Sa(\g)$, \label{associated graded condition}
    \item the projection $\proj_J: \Sa(\g) \rightarrow \Sa(\g/J)$ maps $\gr_d I$ onto $\Sa^d(\g/J)$. \label{rep theory condition}
\end{enumerate}
Then $\ker \Psi$ is generated by $I$ as a two-sided ideal of $\Ua(\g)$.
\end{theorem}

Recall from \Cref{Preliminaries} that for any $A \subseteq \UW$, we denote by $\gr_j A$ the $j$th-graded piece of $\gr A \subseteq \Sa(\W{-1})$, with respect to the order grading defined in \Cref{S Witt filtration def}.

We now define the specific ingredients we will be using in this section. Recall that the symmetrizer map $\sym : \Sa(\W{-1}) \to \Ua(\W{-1})$ is the right inverse of $\gr$ given in \Cref{symmetrizer definition}. 

\begin{definition} \label{conley goode setup def}
We define $d=2$, and the vector subspaces
\begin{align*} J &= \kk\{e_{-1}, e_0, \dots, e_{n-1}\} \subseteq \W{-1},\\
    \mc{J} &= \Ua_1(\W{-1}) \sym (\kk[J]),\\
    I &= (\oo{2n+2}{2n+1}{-1}) \cap \Ua_2(\W{-1}).
\end{align*}
\end{definition}

Of course, we will apply \Cref{conley goode method thm} to the situation $\Psi=\Psi_n$, $T = T_n$. It is clear that conditions \ref{inclusion condition} and \ref{associated graded condition} of \Cref{conley goode method thm} depend only on $\gr \mathcal{J}$ rather than $\mathcal{J}$, and we see below that condition \ref{injective condition} can also be checked on $\gr \mathcal{J}$. Hence, our choice of $\mathcal{J}$ is made simply to ensure
\begin{equation}\label{eq:grJ}
    \gr \mc{J} = \Sa^1(\W{-1})\Sa(J) + \Sa(J) = \kk[\bar{e}_{-1}, \dots, \bar{e}_{n-1}] \oplus \bigoplus_{j=n}^\infty \bar{e}_j \kk[\bar{e}_{-1}, \dots, \bar{e}_{n-1}].
\end{equation}
Notice that the conditions of \Cref{conley goode method thm} also only depend on $\gr I$ rather than $I$.

\begin{lemma} \label{graded injectivity lemma}
    If $\Phi_n$ is injective on $\gr \mathcal{J}$, then $\Psi_n$ is injective on $\mathcal{J}$.
\end{lemma}
\begin{proof}
     As $\Phi_n$ is the associated graded map of $\Psi_n$, by \Cref{filtered algebra hom}, the diagram
    \[\begin{tikzcd}
        \Ua(\W{-1}) \arrow[d, "\gr" '] \arrow[r, "\Psi_n"] & T_n
        \arrow[d, "\gr"] \\
        \Sa(\W{-1}) \arrow[r, "\Phi_n"] & \gr T_n 
    \end{tikzcd}\]
    commutes. Let $z \in \mc{J}$ be nonzero. Then, $\gr(z) \neq 0$. So, if $\Phi_n$ is injective on $\gr \mathcal{J}$,
    \[ \Phi_n(\gr(z)) = \gr(\Psi_n(z)) \neq 0,\]
    implying $\Psi_n(z)$ is non-zero, as required.
\end{proof}

To show that $\Phi_n$ is indeed injective on $\gr \mc{J}$, we will use a change of coordinates described below.

\begin{notation} \label{ci notation}
    Let $\bar{c}_i = \Phi_n(\bar{e_i}) = \gr \Psi_n (e_i) \in \gr T_n =  \kk[\bar{t}, \bar{\del}] \otimes \Sa(\g_n)$, that is,
    \[ \bar{c}_i = \bar{t}^{i+1}\bar{\del} + \sum_{j= 0}^{n-1} \binom{i+1}{j+1} \bar{t}^{i-j} \bar{v}_j .\]
\end{notation}

\begin{notation}\label{not multi index}
    For every $\mb{a} = (a_{-1}, a_{0}, \dots, a_{n-1}) \in \NN^{n+2}$, we define 
    \begin{align*}
        e_\mb{a} &= e^{a_{-1}}_{-1} e^{a_{0}}_{0} \dots  e^{a_{n-1}}_{n-1} \in \Ua(\W{-1}), \\ 
        \bar{c}_\mb{a} &= \bar{c}^{a_{-1}}_{-1} \bar{c}^{a_{0}}_{0} \dots  \bar{c}^{a_{n-1}}_{n-1} \in \gr B_n,
    \end{align*}
    and $|\mb{a}| = a_{-1}+a_{0}+ \dots+ a_{n-1}$, which coincides with the order of $e_\mb{a}$ in $\Ua(\W{-1})$. 
\end{notation}

We now define a useful ring homomorphism and give its values on the $\bar{c}_i$.
\begin{lemma}\label{lem:usefulhom}
Let $\varphi: \gr T_n \to \kk[\bar{t}, \bar{\del}]$ be the algebra homomorphism given by
\[ \varphi(\bar{v_j}) = (-1)^{j+1}\bar{t}^{j+1}\bar{\del} \]
with $\varphi$ the identity on $\kk[\bar{t}, \bar{\del}]$. Then 
\begin{align*}
    \varphi(\bar{c}_i) = \begin{cases}
        \bar{\del} &\text{ if } i = -1 \\ 
        0 &\text{ if } i = 0, \dots, n-1 \\ 
        (-1)^n\binom{i}{n} \bar{t}^{i+1}\bar{\del} &\text{ otherwise.}
    \end{cases} 
\end{align*}
\end{lemma}
\begin{proof}
    We have 
    \begin{align*}
        \varphi(\bar{c}_i) = \bar{t}^{i+1}\bar{\del} + \sum_{j= 0}^{n-1}(-1)^{j+1} \binom{i+1}{j+1} \bar{t}^{i+1}\bar{\del} = \biggl (\sum_{j=0}^{n}(-1)^{j}  \binom{i+1}{j} \biggr )\bar{t}^{i+1}\bar{\del}.
    \end{align*}
    It is easy to check by induction that the following combinatorial identity holds, 
    \[ \sum_{j=0}^{n}(-1)^{i}  \binom{i}{j} = (-1)^{n} \binom{i-1}{n},\]
    from which the lemma follows. 
\end{proof}

We are now prepared to prove the main result of this subsection.

\begin{proposition}\label{prop:phininjective}
    $\Phi_n$ is injective on $\gr(\mc{J})$. 
\end{proposition}
\begin{proof}
Recall that the codomain of $\Phi_n$ is
\[ \gr T_n = \kk[\bar{t}, \bar{\del}, \bar{v}_0, \dots, \bar{v}_{n-1}]. \]
Now, $\bar{t} \in \kk[\bar{t}]\backslash \kk$,  and for $-1 \leq i \leq n-1$, we have
\[ \bar{c}_i \in \kk[\bar{t}, \bar{\del}, \bar{v}_0, \dots, \bar{v}_{i}] \backslash \kk[\bar{t}, \bar{\del}, \bar{v}_0, \dots, \bar{v}_{i-1}].\]
Also, $\bar{c}_{-1} = \bar{\del}$. It follows that $\{\bar{t}, \bar{c}_{-1}, \bar{c}_0, \dots, \bar{c}_{n-1}\}$ is an algebraically independent subset of $\gr T_n$. In fact, recalling \Cref{ci notation}, we may rewrite
\begin{align*}
    \gr T_n = \kk[\bar{t}, \bar{c}_{-1}, \dots, \bar{c}_{n-1}].
\end{align*}
In particular $\{\bar{c}_\mb{a}| \mb{a} \in \NN^{n+1}\}$ is a basis of the vector space
\[ V = \kk[\bar{c}_{-1}, \dots, \bar{c}_{n-1}] .\]

Now, $\gr T_n$ already has two natural gradings (the order grading from \Cref{T filtration def} and the degree grading from \Cref{T grading def}), and we will use an additional filtration in this proof. We tackle this potential clash of terminologies by using the language of degree functions, as follows.

The grading on $\gr T_n$ induced by the order filtration on $T_n$ is given by $\ord$, where
\[ \ord(\bar{t}) = 0, \quad \ord(\bar{\del})=1, \quad \ord(\bar{v}_i)=1, \quad \ord(\bar{c}_i)=1\]
The grading on $\gr T_n$ induced from the degree grading on $T_n$ is given by $\deg$, where
\[ \deg(\bar{t}) = 1, \quad \deg(\bar{\del})=-1, \quad \deg(\bar{v}_i)=i, \quad \deg(\bar{c}_i)=i.\]
We additionally give $\gr T_n$ the ascending filtration given by powers of $\bar{t}$. That is,
\[ \mc{F}^m = \sum_{j=0}^m \kk[\bar{c}_{-1}, \dots, \bar{c}_{n-1}]\bar{t}^j.\]
In particular $\mc{F}^0 = V$. Now, it follows from the description of $\gr \mc{J}$ in (\ref{eq:grJ}) that 
\[ \Phi_n(\gr \mc{J}) = V + \sum_{j=n}^\infty \bar{c}_j V, \]
and it is enough to prove that the right hand side is a direct sum.
Hence, we now describe $\bar{c}_{n+i}$ in terms of $\bar{t}, \bar{c}_{-1}, \dots, \bar{c}_{n-1}$, for $i \geq 0$. Since $\ord(\bar{c}_{n+i})=1$, we have that
\[ \bar{c}_{n+i} \in \kk[\bar{t}]\{\bar{c}_{-1}, \bar{c}_0, \dots, \bar{c}_{n-1}\}. \]
Since $\bar{c}_{j}$ and $\bar{t}$ are degree-homogeneous with $\deg(\bar{c}_{j})=j$ and $\deg(\bar{t})=1$, it follows that
\[ \bar{c}_{n+i} \in \kk \{ \bar{t}^{n+i+1}\bar{c}_{-1}, \bar{t}^{n+i}\bar{c}_0, \dots, \kk \bar{t}^{i+1}\bar{c}_{n-1}\},\]
so let
\[ \bar{c}_{n+i} = \sum_{m=-1}^{n-1} \alpha_m \bar{t}^{n+i-m}\bar{c}_{m},\]
for $\alpha_m \in \kk$. Now, $\varphi(\bar{c}_{n+i}) \neq 0$ but $\varphi(\bar{c}_0) = \dots = \varphi(\bar{c}_{n-1})=0$ by \Cref{lem:usefulhom}. It follows that $\alpha_{-1} \neq 0$, and hence $\bar{c}_{n+i} \in \mc{F}^{n+i+1} \backslash \mc{F}^{n+i} $. Therefore
\[  \bar{c}_{n+i} V \backslash \{0\} \subseteq \mc{F}^{n+i+1}\backslash \mc{F}^{n+i}. \]
It follows that for all $N \geq n$,
\[ V + \sum_{j=n}^N \bar{c}_j V \subseteq \mc{F}^{N+1}\]
but $\bar{c}_{N+1} V \cap \mc{F}^{N+1}=0$. This forces
\begin{align*}
     \left(V + \sum_{j=n}^N \bar{c}_j V \right) \cap \bar{c}_{N+1} V = 0
\end{align*}
for all $N \geq n$, which gives the desired direct sum. Hence $\Phi_n$ is injective on $\gr{\mc{J}}$.
\end{proof}

Combining \Cref{graded injectivity lemma} and \Cref{prop:phininjective} shows that part \ref{injective condition} of \Cref{conley goode method thm} holds.

\begin{proposition} \label{Psi injective on J prop}
    $\Psi_n$ is injective on $\mc{J}$. 
    \qed
\end{proposition}

\subsection{Basic properties of \texorpdfstring{$\mathcal{J}$}{J}} \label{Basic properties of J}
We now check parts \ref{inclusion condition} and \ref{associated graded condition} of \Cref{conley goode method thm}, which are more straightforward. Recall \Cref{conley goode setup def} for the data $d,J, \mathcal{J}$.

\begin{lemma} \label{part b lemma}
    $\Ua_{d-1}(\W{-1}) \subseteq \mathcal{J}$.
\end{lemma}
\begin{proof}
    Follows by definition of $\mathcal{J}$ and $d=2$.
\end{proof}

\begin{lemma} \label{part c lemma}
    For any $m \geq 2$, 
    \begin{align*}
        \gr_m \mathcal{J} = \Sa^{m-d+1}(J)\Sa^{d-1}(\W{-1}) \text{ as a subset of } \Sa(\W{-1}).
    \end{align*}
\end{lemma}
\begin{proof}
    This is a straightforward check, recalling that $\gr_m \mathcal{J}$ refers to the order grading on $\Sa(\W{-1})$ in \Cref{S Witt filtration def}. Let $m \geq 2$. Recall from \eqref{eq:grJ} that 
    \[\gr \mc{J} = \kk[\bar{e}_{-1}, \dots, \bar{e}_{n-1}] \oplus \bigoplus_{j=n}^\infty \bar{e}_j \kk[\bar{e}_{-1}, \dots, \bar{e}_{n-1}].\]
    Thus a basis for $\gr_m \mathcal{J}$ is 
    \[ \{\bar{e}_r \bar{e}_{\mb{a}} \mid r \geq n, |\mb{a}| = m-1\} \cup \{\bar{e}_{\mb{a}} \mid |\mb{a}| = m\}.  \]
    The result follows because $J = \kk \{e_{-1}, \dots, e_{n-1}\}$. 
\end{proof}

\subsection{Representation theory and the projection of \texorpdfstring{$I$}{I}} \label{Representation theory and the projection of I}

We now verify the final condition \ref{rep theory condition} of \Cref{conley goode method thm}. We use analogous reasoning to that of Theorem 7.3 and Theorem 7.4 of \cite{conley2025annihilators}. This requires us to use the structure of tensor density modules and the decomposition of $\Sa^2(\W{-1})$ as a representation of the projective subalgebra $\mf{a} = \{e_{-1}, e_0, e_{1}\} \cong \mf{sl}_2$. A more detailed treatment, can be found in sections 4--6 of \cite{conley2025annihilators}.

\begin{definition}
For $\lambda \in \kk$, we define the tensor density module $\mc{F}_\lambda = \kk[t]$, a $\W{-1}$-representation with 
\begin{align*}
    f\del \cdot g = (fg'+\lambda f'g), \quad \text{ where } f\del \in \W{-1}, g \in \kk[t].
\end{align*}
\end{definition}

\begin{example} \label{adjoint example}
    The adjoint action of $\W{-1}$ on itself is isomorphic to $\mc{F}_{-1}$.
\end{example}

Furthermore, we can extend the adjoint action of \Cref{adjoint example} to make $\UW$ a $\W{-1}$-representation, and we consider this structure for the remainder of the subsection.

\begin{lemma}[{\cite[Proposition 6.3]{conley2025annihilators}}] \label{S2 decomposition}
    There is a unique $\mf{a}$-decomposition
    \begin{align*}
        \sym(\Sa^2(\W{-1})) = {G}_{-2} \oplus {G}_{0} \oplus G_{2} \oplus 
    G_{4} \dots,
    \end{align*}
    where $G_{2\ell}$ is an $\mf{a}$-submodule of $\Ua_2(\W{-1})$ which is $\mf{a}$-isomorphic to $\mc{F}_{2\ell}$.
    \qed
\end{lemma}

\begin{definition}
    For all integer $\ell \geq -1$, let 
    \[ H_{2\ell} = G_{2\ell} \oplus G_{2\ell+2} \oplus \dots .\]
\end{definition}

\begin{lemma}\label{lem: H2n generated}
For $\ell \geq 1$, $H_{2\ell}$ is generated as a $\W{-1}$-module by $\oo{2\ell+2}{2\ell+1}{-1}$.
\end{lemma}
\begin{proof}
    It follows from \cite[Corollary 6.5]{conley2025annihilators} that $H_{2\ell}$ is generated as a $\W{-1}$-module by $S^{\ell+1}(e_{-1}^2)$. The lemma now follows from \Cref{S Omega lemma}.
\end{proof}

\begin{proposition} \label{part d lemma}
Let $\proj_J$ be the projection $\Sa(\W{-1}) \to \Sa(\W{-1}/J)$. Then 
\[ \proj_J (\gr_2 I) = \proj_J (\gr I) = \Sa^2(\W{-1}/J). \]
\end{proposition}
\begin{proof}
Note that $\gr I = \gr_2 I$ because $I \subseteq \Ua_2(\W{-1})$ and $(\oo{2n+2}{2n+1}{-1})$ cannot contain non-zero elements of order less than 2. Notice that parts \ref{injective condition}, \ref{inclusion condition}, \ref{associated graded condition} of \Cref{conley goode method thm} hold by \Cref{Psi injective on J prop}, \Cref{part b lemma} and \Cref{part c lemma} respectively. Thus by \cite[Lemma 3.2]{conley2025annihilators}, $\proj_J$ is injective on $\gr_2 I$. Thus we show that
\[ \proj_J: \gr_2 I \rightarrow \Sa^2(\W{-1}/J)\]
is bijective. In fact, $\proj_J$ respects $e_0$-weight spaces because the degree grading on $\Ua(\W{-1})$ and $\Sa(\W{-1})$ coincides with the $e_0$-weight.

Now, for any given weight $w \in \kk$ there are only finitely many $l \geq -1$ such that $\mc{F}_{2l}$ has a non-zero $e_0$-weight space of weight $w$, and the weight space is always finite-dimensional. On the other hand, $\mc{F}_{2l}$ is spanned by its weight spaces. Thus by \Cref{S2 decomposition}, every weight space of $\Sa^2(\W{-1})$ is finite-dimensional, and together the weight spaces span $\Sa^2(\W{-1})$. The same must then hold for $\Sa^2(\W{-1}/J)$.

Therefore, it suffices to check that $\gr_2 I$ and $\Sa^2(\W{-1}/J)$ have the same weight space dimensions. Because $\gr$ respects $e_0$-weight spaces and $\gr I = \gr_2 I$, it is enough to check that $I$ has the same weight space dimensions as $\Sa^2(\W{-1}/J)$.

To see this, consider the Borel subalgebra $\mf{b} = \kk\{e_{-1}, e_0\}$ of the projective subalgebra $\mf{a}$. 
The $\kk$-linear map
\begin{align*}
    \theta: \W{-1}/J \to \mathcal{F}_n, \quad e_{n+m} +J \mapsto \binom{n+m+1}{m} t^m,
\end{align*}
is an isomorphism of $\mf{b}$-representations. Thus by \cite[Lemma 5.3]{conley2025annihilators}, we have an isomorphism of $\mathfrak{b}$-representations
\[ \Sa^2(\W{-1}/J) \cong \Sa^2(\mathcal{F}_n) \cong \bigoplus_{j \geq n} \mathcal{F}_{2j}.\]
Since $I = H_{2n}$ by \Cref{lem: H2n generated}, it follows that $\Sa^2(\W{-1}/J) \cong I$ as $\mf{b}$-representations, from which the result follows.
\end{proof}

\subsection{\texorpdfstring{\Cref{kernel thm}}{The kernel theorem} and its consequences} \label{Proof of kernel thm}
Our main result now follows by applying \Cref{conley goode method thm}.

\begin{proof}[Proof of \Cref{kernel thm}]
For $\Psi=\Psi_n$ and $T=T_n$, \Cref{Psi injective on J prop}, \Cref{part b lemma}, \Cref{part c lemma}, and \Cref{part d lemma} are conditions \ref{injective condition}-\ref{rep theory condition} of \Cref{conley goode method thm}, respectively.
Thus $\ker(\Psi_n)$ is the two-sided ideal generated by $I$, hence is generated by $\oo{2n+2}{2n+1}{-1}$.
\end{proof}

\Cref{kernel thm} has a range of consequences, the first of which is a straightforward calculation.

\begin{lemma} \label{image of kernel lemma}
    Let $n \geq 1$. The ideal $\Psi_{n+1}(\ker \Psi_n) = T_{n+1}v_n^2$.
\end{lemma}
\begin{proof}
    By \Cref{kernel thm} and \Cref{image of Omega thm}, $\Psi_{n+1}(\ker \Psi_n)$ is the two-sided ideal of $B_{n+1}$ generated by $v_n^2$. Thus by \Cref{two-sided ideal image prop}, $\Psi_{n+1}(\ker \Psi_n) = T_{n+1}v_n^2$, as required.
\end{proof}

This simple observation has two important consequences. The first is that despite $\ker \Psi_n$ being a principal two-sided ideal, it is never finitely generated as a left or right ideal of $\UW$.

\begin{corollary}\label{cor:not fg generated}
    For any $n \geq 1$, $\ker \Psi_n$ is not finitely generated as a left or right ideal of $\UW$.
\end{corollary}
\begin{proof}
    By \Cref{non fg image ideal} and \Cref{image of kernel lemma}, $\Psi_{n+1}(\ker \Psi_n)$ is not finitely generated as a left or right ideal of $B_{n+1}$. It follows that $\ker \Psi_n$ also cannot be finitely generated as a left or right ideal of $\UW$.
\end{proof}

We thus obtain the first main result of this section from the combination of \Cref{kernel thm} and \Cref{cor:not fg generated}.

\kernel*

We compute an explicit set of generators of $\ker \Psi_n$ as a left ideal in the next subsection. Next, we can compute the kernel of the natural map $B_{n+1} \rightarrow B_n$.

\begin{corollary} \label{B_n kernel thm}
    Let $n \geq 1$. The natural surjection $B_{n+1} \rightarrow B_n$ has kernel $T_{n+1}v_{n}^2$.
\end{corollary}
\begin{proof}
    Let $q_n:B_{n+1} \rightarrow B_n$ be the natural surjection corresponding to $\mf{g}_{n+1} \rightarrow \mf{g}_n$. Then $q_n \circ \Psi_{n+1} = \Psi_n$. It follows that
    \[ \ker q_n = \Psi_{n+1}(\ker \Psi_n) = T_{n+1}v_{n}^2, \]
    by \Cref{image of kernel lemma}.
\end{proof}
Because of \Cref{B_n kernel thm}, we have a greater understanding of the image of $\Psi_n$ than is implied by \Cref{Main result large image}. In fact, we see that the image $B_n$ is a successive extension of $B_1$ by the ideals $T_2v_1^2, T_3v_2^2, \dots, T_nv_{n-1}^2$. For example, the general form of an element of $B_n$ is 
\[ s(x) + yv_{n-1}^2, \]
where $x \in B_{n-1}$, $y \in T_n$, and $s:T_{n-1} \rightarrow T_n$ is any homomorphism of vector spaces which is a left inverse to the natural quotient $T_n \rightarrow T_{n-1}$. We suggest that this relation between $B_n$ and $B_1$ is very weak evidence that \Cref{Bn ACC conjecture} should hold.\\

Finally, we can answer a question of Conley and Goode: on page 20 of \cite{conley2025annihilators}, it is asked whether the ideals generated by the elements $S^n(e_{-1}^2)$, generated by repeated application of the step element $S$, are all distinct. It is not clear $\textit{a priori}$ whether this chain of ideals should stabilise or not. However, from \Cref{kernel thm}, and using \Cref{S Omega lemma} as in the proof of \Cref{lem: H2n generated}, we obtain the following.

\begin{corollary}
    For any $n \geq 0$, $(S^n(e_{-1}^2)) = \ker \Psi_n$.
\end{corollary}

The ideals $\ker \Psi_n$ form a descending chain which must be strict, for example because the images $\Psi_n(\Ua(\W{-1}))$ are all of different GK dimension by \Cref{Main result birational image}. In particular, the ideals $(S^n(e_{-1}^2))$ are indeed distinct.

Moreover, \Cref{Main result even odd annihilators} shows that the $(S^{n}(e_{-1}^2))$ are the (intersections) of annihilators of certain simple modules for $\W{-1}$ (corresponding to one-point local functions), see sub\cref{Primitive ideals}.

\subsection{Ideals generated by differentiators} \label{Generation of ker Psi_n}

In this subsection, we show a kind of converse to \Cref{kernel thm}; as well as the kernel of an orbit homomorphism being generated by a differentiator, any (non-zero) differentiator generates the kernel of an orbit homomorphism. As a consequence, we also determine an explicit set of generators of $\ker \Psi_n$ as a one-sided ideal. See \Cref{Main result generic kernel}.

Throughout this subsection, let $n \geq 1$. Since we already know a generator of $\ker \Psi_n$ by \Cref{kernel thm}, we first show an inclusion of ideals.

\begin{lemma}\label{lem: mks part I}
    If $m \geq 2n+1$, $k \geq 2n+1$ and $s \geq -1$, then
    \[ (\oo{m}{k}{s}) \subseteq (\oo{2n+2}{2n+1}{-1}).\]
\end{lemma}
\begin{proof}
    By the equation (16) on page 21 of \cite{conley2025annihilators}, we have
    \[ H_{2n} = \kk\{\oo{m}{k}{s} \mid k \geq m-1, s \geq -1, m \geq 2n+1\}.\]
    By \Cref{lem: H2n generated}, $H_{2n}$ is generated by $\oo{2n+2}{2n+1}{-1}$ as a $\W{-1}$-module. Thus for all $m \geq 2n+1$, $k \geq s$ and $s \geq -1$, $\oo{m}{k}{s}$ is in the $\W{-1}$-module generated by $\oo{2n+2}{2n+1}{-1}$, thus in the ideal of $\Ua(\W{-1})$ generated by $\oo{2n+2}{2n+1}{-1}$.
\end{proof}

In proving the main result of this subsection, we will need some special linear relations between certain differentiators, which the next two lemmas describe.

\begin{lemma}\label{lem: mks part II}
    We have
    \begin{align*}
         \oo{2n+1}{2n+1}{-1} = -\oo{2n+1}{2n}{0} = \frac{1}{2} \oo{2n+2}{2n+1}{-1}.
    \end{align*}
\end{lemma}
\begin{proof}
    By the description of differentiators in \eqref{Omega long form eqn} and reindexing, 
    \begin{align*}
        \oo{2n+1}{2n+1}{-1} &= - \sum_{i=0}^{2n+1} (-1)^i \binom{2n+1}{2n+1-i} e_{i}e_{2n-i}, \quad  
        \oo{2n+1}{2n}{0} &= \sum_{i=0}^{2n+1} (-1)^i \binom{2n+1}{i} e_{2n-i}e_{i}.
    \end{align*}
    Hence 
    \begin{align*}
        \oo{2n+1}{2n+1}{-1} + \oo{2n+1}{2n}{0} 
        &= \sum_{i=0}^{2n+1} (-1)^i \binom{2n+1}{i} ({e}_{2n-i}{e}_{i}-{e}_i{e}_{2n-i}) \\
        &=  \Bigl (\sum_{i=0}^{2n+1} (-1)^i \binom{2n+1}{i} (2i-2n) \Bigr ) e_{2n},
    \end{align*}
    which is zero because of the combinatorial identities 
    \begin{align}\label{comb identity}
        \sum_{i=0}^{2n+1} (-1)^i \binom{2n+1}{i} = 0, \quad \sum_{i=0}^{2n+1} (-1)^i \binom{2n+1}{i}i = 0.
    \end{align}
    Thus 
    \begin{align*}
        \oo{2n+2}{2n+1}{-1} = \oo{2n+1}{2n+1}{-1} - \oo{2n+1}{2n}{0} = 2 \oo{2n+1}{2n+1}{-1} = -2 \oo{2n+1}{2n}{0}.
    \end{align*}
\end{proof}

\begin{lemma}\label{lem: mks part III pre}
    For all $s \geq -1$, we have 
    $$\oo{2n+1}{2n+1+s}{s} = 0.$$
\end{lemma}
\begin{proof}
    Similarly to the proof of \Cref{lem: mks part II}, by \eqref{Omega long form eqn} and reindexing, we have
    \begin{align*}
        2\oo{2n+1}{2n+1+s}{s} &= \sum_{i=0}^{2n+1}(-1)^i\binom{2n+1}{i} (e_{2n+1+s-i}e_{s+i} - e_{s+i}e_{s+ 2n+1-i})\\ 
        &= \Bigl (\sum_{i=0}^{2n+1}(-1)^i\binom{2n+1}{i} (2i-2n-1) \Bigr ) e_{2n+2s+1}, 
    \end{align*}
    which is zero again by \eqref{comb identity}.
\end{proof}

We now prove the main technical fact of this subsection.

\begin{proposition}\label{lem: mks part III}
    Let $m \in \{2n+1, 2n+2\}$, $k \geq m-1$ and $s \geq -1$. 
    Then $\oo{2n+2}{2n+1}{-1}$ lies in the $\W{-1}$-module generated by $\oo{m}{k}{s}$ unless $m=k-s=2n+1$.
\end{proposition}
\begin{proof}
    We give a proof by induction, divided into cases depending on the parity of $m$.
    
    \textbf{Case I}: Suppose $m = 2n+2$. We show that
    \begin{align*}
        \ad(e_{-1})^{k+s-2n} (\oo{m}{k}{s}) \in \mathbb{Z}_{> 0} \oo{2n+2}{2n+1}{-1},
    \end{align*}
    where $\mathbb{Z}_{> 0}$ is the set of positive integers, giving the result since $\kk$ is of characteristic zero. We prove this by induction on $k+s$. The base case $k+s=2n$ is trivial, as we must have $k=2n+1, s=-1$. Applying $\ad(e_{-1})^{k+s-2n-1}$ to the commutation relation \eqref{commutator with e-1} from \Cref{differentiator commutators lemma} gives
    \begin{align*}\ad(e_{-1})^{k+s-2n}( \oo{m}{k}{s}) = &\;(k+1-m)\ad(e_{-1})^{k+s-2n-1}(\oo{m}{k-1}{s}) \\ &+ (s+1) \ad(e_{-1})^{k+s-2n-1}(\oo{m}{k}{s-1}).\end{align*}
    Both coefficients $k+1-m,s+1$ are non-negative integers, and when $k+s > 2n$, at least one is positive. Thus the result follows by induction.

    \textbf{Case II}: Suppose $m = 2n+1$, so $k-s \neq 2n+1$ by assumption. We show that 
    \[
    \ad(e_{-1})^{k+s-2n}(\oo{2n+1}{k}{s}) \in
    \begin{cases*}
    \ZZ_{>0} \oo{2n+1}{2n}{0} & \text{ if } $k-s > 2n+1$ \text{ (Case IIa)},\\
    \ZZ_{>0} \oo{2n+1}{2n+1}{-1} & \text{ if } $k-s < 2n+1$ \text{ (Case IIb)},\\
    \end{cases*}
    \]
    which gives the result by \Cref{lem: mks part II}. We proceed by induction on $k+s$, with trivial base cases $(k,s)=(2n,0), (2n+1, -1)$ respectively.
    
    In the case IIa, suppose $k+s > 2n$. If $(k-1) - s \neq 2n+1$, then by identical reasoning to Case I, induction gives the result. If $(k-1) - s=2n+1$, then  $\oo{m}{k-1}{s}=0$ by \Cref{lem: mks part III pre}, so we have
    \[\ad(e_{-1})^{k+s-2n}( \oo{m}{k}{s}) = (s+1) \ad(e_{-1})^{k+s-2n-1}(\oo{m}{k}{s-1}),\]
    but $s+1 > 0$ because $k+s > 2n$. Thus induction using the equation above from \eqref{commutator with e-1} again gives the result. A very similar argument demonstrates case IIb.
\end{proof}

To illustrate the iterative nature of the above induction proof, the figures below show the effect of successive applications of $\ad(e_{-1})$ to a differentiator $\oo{m}{k}{s}$, by showing the points $(k',s')$ on the $j$th diagonal line $k'+s'=k+s-2n-j$ such that $\oo{m}{k'}{s'}$ appears as a non-zero summand of $\ad(e_{-1})^j \oo{m}{k}{s}$.

\begin{figure}[H]
\centering
\begin{minipage}{.45\textwidth}
  \centering
  \input{tikzpicture/even}
  \caption*{Case I}
\end{minipage}
\begin{minipage}{.45\textwidth}
  \centering
  \input{tikzpicture/odd2}
  \caption*{Case IIa}
\end{minipage}
\end{figure}

We are now ready to prove the main result of this subsection -- in short, if $m \in \{2n+1, 2n+2\}$, then $\ker \Psi_n$ is generated as a left ideal by \emph{all} $m$th order differentiators, but is generated as a two-sided ideal by any one of them.
    
\generickernel*

\begin{proof}
    Notice that \Cref{differentiator commutators lemma} shows that if $i \in \{-1,0,1,2\}$, then
    \[ [e_i, \oo{m}{k_0}{s_0}] \in \sum_{k', s'} \kk \oo{m}{k'}{s'},\]
    and hence
    \[ \oo{m}{k_0}{s_0}e_i \in \UW \{\oo{m}{k}{s} \mid k \geq m-1, s \geq -1\}.\]
    Since $\UW$ is generated by $\{e_{-1}, e_0, e_1, e_2\}$ as a $\kk$-algebra, it follows that
    \[ (\oo{m}{k_0}{s_0}) \subseteq \UW \{\oo{m}{k}{s} \mid k \geq m-1, s \geq -1\}. \]
    However, \Cref{lem: mks part I}, \Cref{lem: mks part III} and \Cref{kernel thm} show that $\ker \Psi_n = (\oo{m}{k}{s})$ for any $k,s$ with $\oo{m}{k}{s} \neq 0$. Thus the above inclusion must be an equality.
\end{proof}

In fact, a careful look at \Cref{lem: mks part III} and \Cref{lem: mks part I} gives us the following result.

\begin{corollary}
    Let $m \in \{2n+1, 2n+2\}$, $k \geq m-1$ and $s \geq -1$. Then $H_{2n}$ is generated by $\oo{m}{k}{s}$ as a $\W{-1}$-module unless $m=k-s=2n+1$.
\end{corollary}

\begin{proof}
    The proof of \Cref{lem: mks part I} implies that $\oo{m}{k}{s} \in H_{2n}$ for all such $m,k, s$. Moreover, by \Cref{lem: mks part III}, $\oo{2n+2}{2n+1}{-1}$ lies in the $\W{-1}$-module generated by $\oo{m}{k}{s}$. But $H_{2n}$ is the $\W{-1}$-module generated by  $\oo{2n+2}{2n+1}{-1}$, thus $H_{2n}$ is the $\W{-1}$-module generated by  $\oo{m}{k}{s}$.
\end{proof}

An important class of representations of the Witt algebra are the \emph{cuspidal} modules, which are those with weight spaces of bounded dimension, see \cite[Definition 2.4]{billig2016classification}. A consequence of \Cref{Main result generic kernel} is a connection between cuspidal modules for $\W{-1}$ and the homomorphic images $B_n$.

\begin{corollary} \label{cuspidal module cor}
    The action of $\UW$ on any finite length cuspidal $W$-module factors through some $\Psi_n$. That is, any finite length cuspidal $W$-module is naturally a $B_n$-module for some $n$.
\end{corollary}
\begin{proof}
    Let $M$ be a cuspidal $W$-module of finite length. Then $M$ is naturally a $\Ua(W)$-module, and hence a $\Ua(\W{-1})$-module by restriction. By Corollary 3.4 of \cite{billig2016classification}, there exists $m$ such that for all $k,s$, the differentiator $\oo{m}{k}{s}$ acts by zero on $M$. Thus if $2n \geq m-1$, \Cref{Main result generic kernel} implies
    \[ (\ker \Psi_n) M =0, \]
    and hence the action of $\UW$ factors through $\Psi_n$, making $M$ naturally a $B_n$-module.
\end{proof}

\subsection{Primitive ideals} \label{Primitive ideals}

Recall that a two-sided ideal is called \emph{primitive} if it is the annihilator of a simple module, or \emph{semi-primitive} if it is an intersection of primitive ideals. In this subsection we show that the kernels of orbit homomorphisms are always primitive or semi-primitive.

First, let us recall the definition of a one-point local function from \cite[Definition 3.0.2]{petukhov2023poisson}.
\begin{definition}\label{def:local function}
    Let $x, \alpha_0, \alpha_1, ...,\alpha_n \in \kk$ with $\alpha_n \neq 0$. We define \emph{a one-point local function} $\chi_{x; \alpha_0, ...,\alpha_n} \in \W{-1}^\ast$ by  
\[
    \chi_{x; \alpha_0, ...,\alpha_n}:  \W{-1} \mapsto \kk, \quad f \partial \mapsto \alpha_0 f(x) + \alpha_1 f'(x) + ... +\alpha_n f^{(n)}(x),
\]
    where $f^{(i)}$ denotes taking the $i$th-derivative of $f$. We call $n$ the \emph{order} of $\chi_{x; \alpha_0, ...,\alpha_n}$. 
\end{definition}

One of the main results of \cite{pham2025orbit} is that there is a ``strong'' Dixmier map from the space of Poisson primitive ideals of $\Sa(\W{-1})$ to the space of primitive ideals of $\Ua(\W{-1})$, \cite[Theorem 1.1]{pham2025orbit}. In fact, this map descends from a natural association of a primitive ideal $Q_\chi$ to every (one-point) local function $\chi$. This means that the primitive ideals $Q_\chi$ are fundamental to understanding of the ring structure of $\UW$.

One application of \Cref{Main result kernel} is to calculate the generators of many of the $Q_\chi$.

\begin{corollary} \label{primitive generators cor}
    If $\chi$ is a one-point local function on $\W{-1}$ of positive even order $2m$, then 
    $$Q_\chi = \ker \Psi_{2m} = (\oo{4m+2}{4m+1}{-1}).$$
    Thus $(\oo{4m+2}{4m+1}{-1})$ is a primitive ideal of $\Ua(\W{-1})$.
\end{corollary}
\begin{proof}
    By \cite[Corollary 6.15]{pham2025orbit}, $Q_\chi = \ker \Psi_{2m}$, thus the result follows from \Cref{Main result kernel}.
\end{proof}

However, for a one-point local function $\chi$ of odd order $2m+1$, $Q_\chi$ need not be equal to $\ker \Psi_{2m+1}$. Nevertheless, $\Psi_{2m+1}$ is still the ``universal homomorphism'' for all local functions of order at most $2m+1$, in the following sense. 

\begin{corollary} \label{semi-primitive generators cor}
    Let $m \geq 0$. Then
    $$\bigcap \{Q_\chi \mid \chi \text{ a one-point local function of order\,} \leq 2m+1  \} = \ker \Psi_{2m+1} = (\oo{4m+4}{4m+3}{-1}).$$
    Thus $(\oo{4m+4}{4m+3}{-1})$ is a semi-primitive ideal of $\Ua(\W{-1})$.
\end{corollary}

\begin{proof}
    By \cite[Corollary 6.11]{pham2025orbit},
    \[ \ker \Psi_{2m+1} = \bigcap \{Q_\chi \mid \chi \text{ a one-point local function of order\,}\leq 2m+1  \},\]
    so the result follows from \Cref{Main result kernel}.
\end{proof}

The combination of \Cref{primitive generators cor,semi-primitive generators cor} is \Cref{Main result even odd annihilators}. Of course, \Cref{Main result generic kernel} also gives a multitude of alternative generators for the primitive and semi-primitive ideals in these results.

A natural question for further research is whether one can calculate the primitive ideal $Q_\chi$ when $\chi$ is a one-point local function of odd order. Motivated by \cite[Theorem 10.1]{conley2025annihilators}, we conjecture that the primitive ideal $Q_\chi$ will be generated by elements of order $2$ and $3$.

\begin{conjecture}
    Let $\chi$ be a one-point local function on $\W{-1}$ of odd order. Then $Q_\chi$ is generated as a two-sided ideal by two elements of order 2,3 respectively.
\end{conjecture}

%% file: tikzpicture/even.tex
\begin{tikzpicture}[scale=0.5]
	\draw[draw=black, -latex, thin, solid] (-3.00,-2.00) -- (-3.00,4.00);
	\draw[draw=black, -latex, thin, solid] (-3.00,-2.00) -- (3.00,-2.00);
	\draw[draw=black, fill=black, thin, solid] (-3.00,-2.00) circle (0.1);
	\draw[draw=black, fill=black, thin, solid] (-3.00,-1.00) -- (-2.00,-2.00);
	\draw[draw=black, fill=black, thin, solid] (-3.00,0.00) -- (-1.00,-2.00);
	\draw[draw=black, fill=black, thin, solid] (-3.00,1.00) -- (0.00,-2.00);
	\draw[draw=black, fill=black, thin, solid] (-3.00,2.00) -- (0.00,-1.00);
	\draw[draw=black, fill=black, thin, solid] (-2.00,2.00) -- (0.00,0.00);
	\draw[draw=black, fill=black, thin, solid] (-1.00,2.00) -- (0.00,1.00);
	\draw[draw=black, fill=black, thin, solid] (0.00,2.00) -- (-3.00,2.00);
	\draw[draw=black, fill=black, thin, solid] (0.00,2.00) -- (0.00,-2.00);
	\draw[draw=black, fill=black, thin, solid] (-3.00,-1.00) circle (0.1);
	\draw[draw=black, fill=black, thin, solid] (-2.00,-2.00) circle (0.1);
	\draw[draw=black,fill=black, thin, solid] (-1.00,-2.00) circle (0.1);
	\draw[draw=black, fill=black,thin, solid] (-2.00,-1.00) circle (0.1);
	\draw[draw=black,fill=black, thin, solid] (-3.00,0.00) circle (0.1);
	\draw[draw=black,fill=black, thin, solid] (-3.00,1.00) circle (0.1);
	\draw[draw=black,fill=black, thin, solid] (-2.00,0.00) circle (0.1);
	\draw[draw=black,fill=black, thin, solid] (-1.00,-1.00) circle (0.1);
	\draw[draw=black,fill=black, thin, solid] (0.00,-2.00) circle (0.1);
	\draw[draw=black,fill=black, thin, solid] (0.00,-1.00) circle (0.1);
	\draw[draw=black,fill=black, thin, solid] (-1.00,0.00) circle (0.1);
	\draw[draw=black,fill=black, thin, solid] (-2.00,1.00) circle (0.1);
	\draw[draw=black,fill=black, thin, solid] (-3.00,2.00) circle (0.1);
	\draw[draw=black,fill=black, thin, solid] (-2.00,2.00) circle (0.1);
	\draw[draw=black,fill=black, thin, solid] (-1.00,1.00) circle (0.1);
	\draw[draw=black,fill=black, thin, solid] (0.00,0.00) circle (0.1);
	\draw[draw=black,fill=black, thin, solid] (0.00,1.00) circle (0.1);
	\draw[draw=black,fill=black, thin, solid] (-1.00,2.00) circle (0.1);
	\draw[draw=black, fill=black, thin, solid] (0.00,2.00) circle (0.1);
    \node[black] at (1,2.25){\tiny $(k, s)$};
    \node[black] at (-1.00,2.5) {\tiny $(k-1, s)$};
    \node[black] at (1.5,1.00) {\tiny $(k, s-1)$};
    \node[black] at (-3.5,-2.5) {\tiny $(2n+1, -1)$};
    \node[black] at (3,-2.5) {\tiny $k$}; 
    \node[black] at (-3.5,4)  {\tiny $s$}; 
\end{tikzpicture}

%% file: tikzpicture/odd2.tex
\begin{tikzpicture}[scale=0.5]
	\draw[draw=black, -latex, thin, solid] (-3.00,-2.00) -- (3.00,-2.00);
	\draw[draw=black, -latex, thin, solid] (-3.00,-2.00) -- (-3.00,4.00);
    
	\draw[draw=black, -latex, thin, dashed] (-3.00,-2.00) -- (3.00,4.00);
	\draw[draw=black, fill=lightgray, thin, solid] (-3.00,-2.00) circle (0.1);
	\draw[draw=black, fill=lightgray,thin, solid] (-2.00,-1.00) circle (0.1);
    \draw[draw=black,fill=lightgray, thin, solid] (-1.00,0.00) circle (0.1);
    \draw[draw=black,fill=lightgray, thin, solid] (0.00,1.00) circle (0.1);

	\draw[draw=black, fill=black, thin, solid] (1.00,0.00) circle (0.1);
	\draw[draw=black, fill=black, thin, solid] (0.00,0.00) circle (0.1);
	\draw[draw=black, fill=black, thin, solid] (1.00,-1.00) circle (0.1);
	\draw[draw=black, fill=black, thin, solid] (0.00,-1.00) circle (0.1);
	\draw[draw=black, fill=black, thin, solid] (1.00,-2.00) circle (0.1);
	\draw[draw=black, fill=black, thin, solid] (0.00,-2.00) circle (0.1);
	\draw[draw=black, fill=black, thin, solid] (-1.00,-1.00) circle (0.1);
	\draw[draw=black, fill=black, thin, solid] (-1.00,-2.00) circle (0.1);
	\draw[draw=black, fill=black, thin, solid] (-2.00,-2.00) circle (0.1);
	\draw[draw=black, thin, solid] (0.00,0.00) -- (1.00,-1.00);
	\draw[draw=black, thin, solid] (0.00,-1.00) -- (1.00,-2.00);
	\draw[draw=black, thin, solid] (-1.00,-1.00) -- (0.00,-2.00);
	\draw[draw=black, thin, solid] (1.00,0.00) -- (1.00,-2.00);
	\draw[draw=black, thin, solid] (1.00,0.00) -- (0.00,0.00);

    \node[black] at (1.5,0.5){\tiny $(k, s)$};
    \node[black] at (-3.5,-2.5) {\tiny $(2n, -1)$};
    \node[black] at (3,-2.5) {\tiny $k$};
    \node[black] at (-3.5,4)  {\tiny $s$};
    \node[black] at (4.5,4.5)  {\tiny $k-s = 2n+1$};
\end{tikzpicture}

%% file: main.bbl
\providecommand{\bysame}{\leavevmode\hbox to3em{\hrulefill}\thinspace}
\providecommand{\MR}{\relax\ifhmode\unskip\space\fi MR }
\providecommand{\MRhref}[2]{%
  \href{http://www.ams.org/mathscinet-getitem?mr=#1}{#2}
}
\providecommand{\href}[2]{#2}
\begin{thebibliography}{NvO82}

\bibitem[AM25]{andruskiewitsch2025noetherian}
Nicolás Andruskiewitsch and Olivier Mathieu, \emph{{Noetherian enveloping algebras of simple graded Lie algebras}}, Journal of the Mathematical Society of Japan (2025), \href{https://doi.org/10.2969/jmsj/93619361}{DOI:10.2969/jmsj/93619361}.

\bibitem[AS74]{amayo1974infinite}
Ralph~K. Amayo and Ian Stewart, \emph{{Infinite-dimensional Lie algebras}}, Noordhoff International Publishing, Leyden, 1974.

\bibitem[BF16]{billig2016classification}
Yuly Billig and Vyacheslav Futorny, \emph{{Classification of irreducible representations of Lie algebra of vector fields on a torus}}, Journal für die reine und angewandte Mathematik \textbf{720} (2016), 199--216.

\bibitem[Buz24]{buzaglo2024lie}
Lucas Buzaglo, \emph{Lie algebras of derivations and their universal enveloping algebras}, Ph.D. thesis, University of Edinburgh, 2024, \href{http://dx.doi.org/10.7488/era/4901}{DOI:10.7488/ERA/4901}.

\bibitem[CG25]{conley2025annihilators}
Charles~H. Conley and William Goode, \emph{An approach to annihilators in the context of vector field {L}ie algebras}, Expo. Math. \textbf{43} (2025), no.~2, Paper No. 125600, 23.

\bibitem[CM07]{conley2007}
Charles~H. Conley and Christiane Martin, \emph{Annihilators of tensor density modules}, J. Algebra \textbf{312} (2007), no.~1, 495--526.

\bibitem[Dix96]{dixmier1996enveloping}
Jacques Dixmier, \emph{{Enveloping algebras}}, Graduate Studies in Mathematics, vol.~11, {American Mathematical Society, Providence, RI}, 1996.

\bibitem[IS20]{iyudu2020enveloping}
Natalia~K. Iyudu and Susan~J. Sierra, \emph{{Enveloping algebras with just infinite Gelfand–Kirillov dimension}}, Arkiv för Matematik \textbf{58} (2020), no.~2, 285--306.

\bibitem[LLZ15]{liu2015class}
Genqiang Liu, Rencai Lu, and Kaiming Zhao, \emph{{A class of simple weight Virasoro modules}}, Journal of Algebra \textbf{424} (2015), 506--521.

\bibitem[Mat89]{Matsumura1989}
Hideyuki Matsumura, \emph{Commutative ring theory}, second ed., Cambridge Studies in Advanced Mathematics, vol.~8, Cambridge University Press, Cambridge, 1989, Translated from the Japanese by M. Reid.

\bibitem[Mic73]{Mickelsson1973}
Jouko Mickelsson, \emph{Step algebras of semi-simple subalgebras of {L}ie algebras}, Rep. Mathematical Phys. \textbf{4} (1973), 307--318.

\bibitem[MR01]{mcconnell2001noncommutative}
J.~C. McConnell and J.~C. Robson, \emph{Noncommutative {N}oetherian rings}, revised ed., Graduate Studies in Mathematics, vol.~30, American Mathematical Society, Providence, RI, 2001, With the cooperation of L. W. Small.

\bibitem[NvO82]{nastasescu1982}
C.~N\u{a}st\u{a}sescu and F.~van Oystaeyen, \emph{Graded ring theory}, North-Holland Mathematical Library, vol.~28, North-Holland Publishing Co., Amsterdam-New York, 1982.

\bibitem[Pha25]{pham2025orbit}
Tuan~Anh Pham, \emph{The orbit method for the {V}irasoro algebra}, arXiv preprint (2025), \href{https://arxiv.org/abs/2504.14670}{arXiv:2504.14670}.

\bibitem[PS20]{petukhov2020ideals}
Alexey~V. Petukhov and Susan~J. Sierra, \emph{Ideals in the enveloping algebra of the positive {W}itt algebra}, Algebras and Representation Theory \textbf{23} (2020), 1569--1599.

\bibitem[PS23]{petukhov2023poisson}
Alexey~V. Petukhov and Susan~J. Sierra, \emph{The {P}oisson spectrum of the symmetric algebra of the {V}irasoro algebra}, Compos. Math. \textbf{159} (2023), no.~5, 933--984.

\bibitem[SW14]{SierraWalton}
Susan~J. Sierra and Chelsea Walton, \emph{The universal enveloping algebra of the {W}itt algebra is not noetherian}, Adv. Math. \textbf{262} (2014), 239--260.

\bibitem[SW16]{sierra2016maps}
\bysame, \emph{Maps from the enveloping algebra of the positive {W}itt algebra to regular algebras}, Pacific J. Math. \textbf{284} (2016), no.~2, 475--509.

\bibitem[vdH75]{vdHombergh1975}
A.~van~den Hombergh, \emph{A note on {M}ickelsson's step algebra}, Indag. Math. \textbf{37} (1975), 42--47, Nederl. Akad. Wetensch. Proc. Ser. A {\bf 78}.

\bibitem[War84]{warfield1984gk}
Robert~B. Warfield, Jr., \emph{The {G}el'fand-{K}irillov dimension of a tensor product}, Math. Z. \textbf{185} (1984), no.~4, 441--447.

\bibitem[Zha96]{zhang1996}
James~J. Zhang, \emph{On {G}el'fand-{K}irillov transcendence degree}, Trans. Amer. Math. Soc. \textbf{348} (1996), no.~7, 2867--2899.

\bibitem[Zhe90]{Zhelobenko1990}
D.~P. Zhelobenko, \emph{An introduction to the theory of {$S$}-algebras over reductive {L}ie algebras}, Representation of {L}ie groups and related topics, Adv. Stud. Contemp. Math., vol.~7, Gordon and Breach, New York, 1990, pp.~155--221.

\end{thebibliography}
